\documentclass{article}

\usepackage{amsfonts,amsthm,mathtools,enumerate,esint,subcaption, amssymb}
\usepackage{mathrsfs} 
\usepackage[normalem]{ulem}
\newtheorem*{theorem*}{Theorem}
\newtheorem{theorem}{Theorem}
\newtheorem{lemma}[theorem]{Lemma}

\newtheorem{corollary}[theorem]{Corollary}

\newtheorem*{example*}{Example}
\newtheorem*{remark*}{Remark}
\newtheorem{question}{Question}
\newtheorem{theorema}{Theorem}
\newtheorem{corollaria}[theorema]{Corollary}

\mathtoolsset{showonlyrefs}

\def\R{\mathbb{R}}
\def\T{\mathbb{T}}
\def\Z{\mathbb{Z}}

\def\C{\mathcal{C}}
\def\G{\Gamma}
\def\g{\gamma}
\def\f{\varphi}
\def\t{\theta}

\def\e{\varepsilon}

\def\U{\mathcal U}
\def\D{\mathcal D}
\def\A{\mathbb A}

\usepackage{xcolor}
\definecolor{dgreen}{rgb}{0.1,0.6,0.1}

\definecolor{bluegray}{rgb}{0.4, 0.6, 0.8}

\begin{document}

\title{Existence and Nonexistence of Invariant Curves of Coin Billiards}

\author{Santiago Barbieri\footnote{Universitat Politècnica de Catalunya, Barcelona, Spain \newline email: santiago.barbieri@upc.edu} \and Andrew Clarke\footnote{Universitat Politècnica de Catalunya, Barcelona, Spain \newline email: andrew.michael.clarke@upc.edu} \footnote{Centre de Recerca Matemàtica, Bellaterra, Spain}}

\date{}

\maketitle

\begin{abstract}
In this paper we consider the \emph{coin billiard} introduced by M. Bialy. It is a modification of the classical billiard, obtained as the return map of a nonsmooth geodesic flow on a cylinder that has homeomorphic copies of a classical billiard on the top and on the bottom (a \emph{coin}). The return dynamics is described by a map $T$ of the annulus $\A = \T \times (0,\pi)$. We prove the following three main theorems: in two different scenarios (when the height of the coin is small, or when the coin is near-circular) there is a family of KAM curves close to, but not accumulating on, the boundary $\partial \A$; for any noncircular coin, if the height of the coin is sufficiently large, there is a neighbourhood of $\partial \A$ through which there passes no invariant essential curve; and the only coin billiard for which the phase space $\A$ is foliated by essential invariant curves is the circular one. These results provide partial answers to questions of Bialy. Finally, we describe the results of some numerical experiments on the elliptical coin billiard. 
\end{abstract}

\section{Introduction}

At a conference at CIRM in Luminy in 2021 the following modification of the classical billiard was suggested by Bialy \cite{bialy2022open}. Denote by $\G \subset \R^2$ a compact strictly convex domain, such that $\partial \G$ is $C^{\infty}$ (e.g. an ellipse). Consider the nonsmooth geodesic flow on the cylinder $\C \left(\G, \ell \right)$ of height $\ell$ which has homeomorphic copies of $\G$ glued on the top and bottom. On the top and bottom of the surface $\C \left(\G, \ell \right)$ the motion is along straight lines, and on the cylindrical part it is along geodesics of the cylinder. 

The motion can be described as follows. We introduce an arclength parametrisation $\g = \g(\f)$ on $\partial \G$ and denote by $\A$ the phase cylinder with coordinates $(\f, \t)$, where $\t \in (0, \pi)$. A particle starts at $\g(\f)$ on the top of $\C \left(\G, \ell \right)$ and moves along the straight line making an angle $\t$ with $\g'(\f)$ until it comes to $\partial \G$ again with an angle $\t_1$ at $\g(\f_1)$. Then it passes to the surface of the cylinder with the same angle $\t_1$, and travels along the geodesic of the cylinder until it hits the bottom at $\g(\f+\ell \, \cot \t_1)$ with the same angle $\t_1$. Next it passes to the bottom domain with the angle $\t_1$ and so on (see Figure \ref{figure_coin} in Section \ref{sec_coin}). 

This geodesic motion can be described as a map $T$ of the phase cylinder $\A$ via
\[
T: (\f, \t) \mapsto (\f_1 + \ell \, \cot \t_1, \t_1)
\]
which means that $T$ is the composition $T=T_2 \circ T_1$ where $T_1 : (\f, \t) \mapsto (\f_1, \t_1)$ is the usual billiard and $T_2: (\f_1, \t_1) \mapsto (\f_1 + \ell \, \cot \t_1, \t_1)$ is a shift of the $\f$ coordinate. Note that $T$ is symplectic, and that the shift $\ell \, \cot \t_1$ is unbounded as $(\f_1, \t_1)$ approaches $\partial \A$. In \cite{bialy2022open}, Bialy asked the following three questions:

\begin{question}\label{question1}
	Are there invariant curves of $T$? For example, are there KAM curves near the boundary?
\end{question}

\begin{question}\label{question2}
	What are the shapes of $\partial \Gamma$ (other than circles) such that $T$ is an integrable map?
\end{question}

\begin{question}\label{question3}
	Can $T$ be ergodic?
\end{question}

We refer to the configuration cylinder $\C \left(\G, \ell \right)$ as a \emph{coin}, and to the map $T$ as the \emph{coin map}. In general, we refer to the problem as a \emph{coin billiard}\footnote{Note that the terminology `puck billiard' has also been suggested \cite{drivas2024pensive}.}. Aside from the inherent curiosity of these questions, and the fact that they have been stated as open problems related to billiards and geometrical optics, we offer the following two considerations as further motivation to study coin billiards. 

\begin{enumerate}[(I)]
	\item
	Birkhoff wrote of the classical billiard described by the map $T_1$: \emph{This system is very illuminating for the following reason. Any Lagrangian system with two degrees of freedom is isomorphic with the motion of a particle on a smooth surface rotating uniformly about a fixed axis and carrying a conservative field of force with it. In particular if the surface is not rotating and if the field of force is lacking, the paths of the particle will be geodesics. If the surface is now flattened to the form of a plane convex curve $C$, the `billiard ball problem' results} \cite{birkhoff1927on}. Focusing on the latter part of this process, when the surface is smooth, it is not a straightforward task to realise the flattening to a plane curve with equations which can readily be analysed. Considering instead the nonsmooth geodesic flow on the coin $\C \left(\G, \ell \right)$, we now have a homeomorphic copy of the sphere and a map with precise equations depending on only one parameter $\ell >0$, such that when $\ell \to 0$, we obtain the classical billiard map $T_1$ as described by Birkhoff. 
	\item
	The coin billiard is a special case of the so-called \emph{pensive billiard}, recently defined in \cite{drivas2024pensive}, and based on \cite{khesin2022golden}. These systems arise naturally in the setting of hydrodynamics, and can be used to describe the appearance of the golden ratio in the motion of vortices. 
\end{enumerate}

\subsection*{Main Results}

In this paper, we provide partial answers to Questions \ref{question1}-\ref{question3}. In what follows we give qualitative statements of our main results. The more detailed statements can be found in the relevant sections later in the paper. 

With regards to Question \ref{question1}, we have found KAM curves in two different perturbative settings: where the height of the coin is small; and when the domain is near-circular. 

\begin{theorema}\label{theorem_a}
	Consider the coin mapping $T$ on the coin $\C \left(\G, \ell \right)$, where $\G \subset \R^2$ is a bounded and strictly convex domain with real-analytic boundary, and $\ell >0$. Suppose either: 
	\begin{enumerate}[(a)]
		\item \label{item_kama}
		The height of the coin is perturbative: $\ell \ll 1$; or
		\item \label{item_kamb}
		The curvature is close to constant: $\partial \G$ is close to a circle. 
	\end{enumerate}
	Then there is a strip in $\A$ close to (but not accumulating on) $\partial \A$ that contains a set of positive Lebesgue measure of $T$-invariant curves with Diophantine rotation numbers. 
\end{theorema}

The formal statements of these results can be found in Theorems \ref{theorem_nckam} and \ref{theorem_shkam} in Section \ref{sec_kam} below. The idea of the proof is very simple: we prove that, in each of these perturbative settings, we can make a coordinate transformation in a strip so that the map takes the form $(x,y) \to (x+y, y) + O(\e)$. We then apply Moser's version of the KAM Theorem. Although Theorem \ref{theorem_a} requires the boundary to be real-analytic, it is likely true for much less regular boundaries. The reason we state real-analyticity as the required regularity is to allow us to apply a straightforward version of the KAM theorem (Moser's original theorem, described in the appendix) in order to simplify the exposition. A version of the KAM theorem with lower regularity requirements likely could be applied, but there may be other more cumbersome conditions to check, which would detract from the simple nature of the proof. 

Theorem \ref{theorem_a} is about \emph{existence} of invariant curves. The following theorem is about \emph{nonexistence} of invariant curves in a neighbourhood of the boundary of the phase space $\A$. 

\begin{theorema}\label{theorem_b}
	Let $\G \subset \R^2$ be a bounded and strictly convex domain with $C^5$ boundary. Suppose moreover that $\G$ is not a disc. Then there exists an $\ell_0=\ell_0(\G) >0$ such that for any $\ell > \ell_0$, there is $\delta >0$ for which no invariant essential curve of the coin mapping $T$ on the coin $\C \left( \G, \ell \right)$ passes through the neighbourhood $\left( \T \times (0, \delta) \right) \cup \left( \T \times (\pi - \delta, \pi) \right)$ of $\partial \A$. \end{theorema}

We give two proofs of Theorem \ref{theorem_b}. The first one is variational, and makes use of a particular coordinate transformation in strips accumulating on $\partial \A$. In these coordinates, the coin mapping $T$ is a small perturbation of a generalised standard map. We write a generating function for this map, and combine this with a technique introduced by Mather to prove the absence of invariant curves \cite{mather1982glancing}. In the second proof it is shown that, when $\ell>\ell_0$, if essential invariant curves were to exist near the boundary, their Lipschitz constants would not respect the bounds found by Herman in \cite{herman1983courbes}. A future goal is to establish similar results in the case of more general pensive billiards (i.e. with different delay functions \cite{drivas2024pensive}), which is why we have been considering different approaches to the problem.

Naturally, we would like to provide a complement to Theorem \ref{theorem_b} for the case where the height of the coin is small. As we will discuss in more detail later (see the discussion on Bialy's second question below) this is quite difficult, as it is closely related to proving nonintegrability of a generalised standard map for small values of the parameter, i.e. a notoriously difficult problem that is still open in its full generality.

Upon consideration of Theorems \ref{theorem_a} and \ref{theorem_b}, one may wonder if the coin map on $\C(\G, \ell)$ admits any essential invariant curve at all when the table $\Gamma$ is far from circular and $\ell$ is sufficiently large. Indeed, for certain values of $\t_* \in (0,2\pi)$, Gutkin constructed in \cite{gutkin2012capillary} noncircular analytic boundaries $\G_{*}$ for which the corresponding billiard map $T_1$ has an invariant curve $C_{*} = \{ (\f, \t) : \t = \t_* \}$ (i.e. a circular caustic). Since the map $T_2$ does not alter $\t$, the invariant curve $C_{*}$ remains invariant for the coin map $T$ on $\C \left( \G_{*}, \ell \right)$ for \emph{any} value of $\ell >0$.

Finally, the third main result of this work is the following.

\begin{theorema}\label{theorem_c} 
	Let $\G \subset \R^2$ be any bounded strictly convex domain with $C^5$ boundary. For any $\ell>0$, if the map $T$ associated to the coin billiard $\mathcal C(\G,\ell)$ admits a family of essential invariant curves whose union is
	the annulus $\A$, then $\G$ is a disc.
\end{theorema}

This result is an extension to coin billiards of a Theorem of Bialy about classical billiards. See the original paper \cite{bialy1993circular} for a variational proof, \cite{wojtkowski1994circular} for a geometric approach, and \cite{marcoentropy} for its relation to the Birkhoff conjecture and to questions related to entropy. The proof of Theorem \ref{theorem_c} is essentially different from those that are found in the context of classical billiards, which rely either on variational or on geometric methods. As we will see in Section \ref{sec_birkhoff} we use in an essential way the strong twist property of $T$ near the boundary, and the bounds on the Lipschitz constants that any hypothetical essential invariant curve of $T$ should satisfy, according to Herman \cite{herman1983courbes}. In particular, it is shown that, for noncircular tables, Herman's bounds are incompatible with the existence of a foliation of the phase space near the boundary by essential invariant curves, due to the strong twist property of the coin map in that region.  Incidentally, the same kind of arguments are used in the second proof of Theorem \ref{theorem_b}. We also observe that another reason why Theorem \ref{theorem_c} is interesting is that it holds for any height of the coin billiard. As we will discuss later in the paper, we conjecture that some aspects of its proof might be used in order to prove nonintegrability for small heights of the coin.

	The proof of Theorem \ref{theorem_c} presents the following trichotomy: either $\G$ is a disc; or the coin map $T$ in $\C(\G,\ell)$ is not integrable; or else $\G$ is not a disc, but the coin map $T$ in $\C(\G,\ell)$ is integrable, and there are chains of (likely elliptic) islands, surrounded by separatrices of integer rotation number, that accumulate on $\partial \A$. We expect the latter configuration to be very fragile, as a generic perturbation should break the separatrices of integer rotation number (see e.g. \cite{ramirezrosbreakinvariant}). This could be used in future works to prove  nonintegrability of the coin billiard when the height is small. 

In Theorems \ref{theorem_b} and \ref{theorem_c} we require that the boundary be $C^5$, for the following reason. In the proofs of these theorems, we first perform an expansion of the coin map $T(\f,\t)=(\bar \f, \bar \t)$ in powers of $\t$ (see equation \eqref{eq_coinexp} below). The term of order $\t^2$ in the expansion of $\bar \t$ depends on the first derivative $\rho'$ of the radius of curvature of $\partial \G$ with respect to $\f$ (i.e. the third derivative of $\partial \G$), and similarly the term of order $\t^3$ depends on $\rho'$ and on $\rho''$. We then need to work with the $\f$ derivative of $\bar \t$; in particular, we work with its Taylor polynomial up to terms of order $\t^2$, and we need that the higher order terms are well-behaved in the limit as $\t \to 0$. This holds when $\partial \G$ is $C^5$.


In what follows we describe some important consequences of Theorems \ref{theorem_b} and \ref{theorem_c}. The statement below is a consequence of results by Mather \cite{matheralphaomega} and Le Calvez \cite{lecalvez1987proprietes}. 

\begin{corollaria}\label{cor_matherlecalvez}
	Under the hypotheses of Theorem \ref{theorem_b}, there exists an orbit whose $\alpha$-limit set is contained in $\A$, and whose $\omega$-limit set is contained in $\partial \A$.
\end{corollaria}


Moreover, other classical results about twist maps of the annulus due to Angenent,  Boyland and Hall, and Katok ensure the following \cite{angenententropy,boylandhallentropy,katokhorseshoe}. 
\begin{corollaria}\label{cor_angenentkatok}
	Under the hypotheses of Theorem \ref{theorem_b} the coin map has a horseshoe, and has positive topological entropy. 
\end{corollaria} 

As a direct consequence of Corollaries \ref{cor_matherlecalvez} and \ref{cor_angenentkatok}, we obtain the following general nonintegrability result for noncircular coins, whenever the height of the coin is sufficiently large. 

\begin{corollaria}
	Let $\G \subset \R^2$ be a bounded and strictly convex domain with $C^5$ boundary. Suppose $\partial \G$ is noncircular. Then for $\ell>\ell_0$, with $\ell_0$ as in Theorem \ref{theorem_b}, the coin mapping $T$ associated to the coin $\mathcal C(\G,\ell)$ is nonintegrable.
\end{corollaria}

In addition to these analytical theorems, we have performed some numerical experiments with elliptical coin billiards. These simulations are described in Section \ref{sec_num}. The phase portraits that we provide clearly illustrate the presence and absence of invariant curves described by Theorems \ref{theorem_a}, \ref{theorem_b}, and \ref{theorem_c} as we vary parameters. 

\subsection*{Responses to Bialy's Questions and Comparison with Classical Billiards}

We now describe the significance of Theorems \ref{theorem_a}, \ref{theorem_b}, and \ref{theorem_c} in the context of Questions \ref{question1}, \ref{question2}, and \ref{question3}. Moreover, we compare the dynamical behaviour of the map on the coin $\C \left( \G, \ell \right)$ with the case when $\ell=0$: the classical billiard. 

\begin{enumerate}
	\item
	\textbf{Invariant Curves:} The main results of this paper regard existence and nonexistence of invariant curves of $T$. In particular, we have established the existence of some invariant curves in two perturbative settings, as well as the nonexistence of invariant curves for any noncircular coin billiard, when the height $\ell$ of the coin is sufficiently large. This provides a partial answer to Bialy's first question. As mentioned previously, in addition to the KAM curves we construct, there are also the special cases of Gutkin tables in which $T_1$-invariant curves of constant height correspond to $T$--invariant curves of constant height \cite{gutkin2012capillary}. 
	
	The nonexistence of invariant curves near the boundary for $\ell \gg 0$ in the noncircular case is in contrast with the situation in classical billiards. Indeed, a celebrated result of Lazutkin implies that classical billiards (sufficiently smooth, planar, strictly convex) \emph{always} have a family of invariant curves accumulating on the boundary \cite{lazutkin1973existence}. In fact, in case \eqref{item_kama} of Theorem \ref{theorem_a}, it can be seen that the strip of the phase space containing $T$-invariant curves tends towards $\partial \A$ as $\ell \to 0$ (see Theorem \ref{theorem_shkam}). Therefore in this sense, we see that the limiting behaviour of the classical billiard can be observed in the coin billiard as we flatten the coin $\C \left( \G, \ell \right)$. 
	\item
	\textbf{Integrability:} We point out that if $\mathbb D$ is a disc then the coin billiard on $\C(\mathbb D, \ell)$ is integrable. Indeed, the angle of reflection $\t$ is a constant of motion. On one hand, Theorem \ref{theorem_b} implies that if $\G$ is \emph{any} noncircular billiard, and $\ell > \ell_0(\G)$ then the coin billiard on $\C(\G, \ell)$ is not integrable. On the other hand, Theorem \ref{theorem_c} states that for any height of the coin the \emph{only} billiard whose phase space is completely foliated by essential invariant curves is $\C(\mathbb D,\ell)$. This provides an almost complete classification of the integrable coin billiards. However $\ell_0(\G)$ is strictly positive whenever $\G$ is not a disc, and while it seems very unlikely that there could be any $\ell>0$ for which $\C(\G, \ell)$ is integrable with $\G$ different from a disc, we have not yet been able to obtain a proof of this. Indeed, close to the boundary $\partial \A$, the map can be reduced to a generalised standard map, with parameter $\ell$. The question of nonintegrability of generalised standard maps for small values of the parameter is a notorious open problem plagued by the exponentially small splitting of separatrices, and was solved in the special case of the Chirikov standard map by Lazutkin and his school only after a tremendous effort over many years \cite{lazutkin2005splitting}. However, as we shall see in Section \ref{sec_birkhoff}, in the proof of Theorem \ref{theorem_c} it is demonstrated that, for any $\ell>0$, if $\C(\G,\ell)$ is integrable, then there exist separatrices of integer rotation number that bound (likely elliptic) islands accumulating on the boundary $\partial \A$. This kind of configuration appears to be very fragile, and we expect that the separatrices should generically split, thus giving rise to chaos. We point out that, although the coin map is closely related to the standard map, this is only after performing a local change of variables that effectively hides the strong twist condition. It is possible that the strong twist condition could itself be used to prove nonintegrability via different methods, thus avoiding problems of exponential smallness.
	
	By contrast, in the classical setting, the only known integrable billiard is in an ellipse (with the circle being a special case). Indeed, this is famously conjectured to be the only integrable classical billiard, a problem which remains open in its full generality, and is called the \emph{Birkhoff conjecture}. 
	\item
	\textbf{Ergodicity:} Clearly, in the case of homeomorphisms of a two-dimensional annulus, the presence of invariant essential curves implies that the system is not ergodic, as such curves necessarily divide the phase space into invariant regions of strictly positive, but not full, Liouville measure. As a result, our Theorem \ref{theorem_a} implies that if $\ell$ is small or if $\G$ is close to a disc, then the coin mapping on $\C \left( \G, \ell \right)$ is not ergodic. 
	
	As a comparison with classical billiards, a famous result of Lazutkin implies the existence of invariant curves near the boundary in \emph{any} strictly convex and sufficiently smooth billiard \cite{lazutkin1973existence}, so the classical billiard is not ergodic. Therefore, with respect to ergodicity, the coin $\C \left( \G, \ell \right)$ for small values of $\ell$ and the limiting billiard $\G$ agree. 
	
	Essential invariant curves are not, however the only obstacle to ergodicity. Elliptic islands around elliptic periodic orbits, whenever they exist, are nontrivial invariant sets of positive measure. As an example of when they can exist, we can consider the case when $\G$ is an ellipse (see Section \ref{sec_num} for details). In general, it is not easy to rule out the existence of elliptic periodic orbits for a concrete model, but if this could be done, then for sufficiently large values of $\ell$, it may be possible to observe ergodicity. 
\end{enumerate}

\subsection*{Structure of the Paper}

The structure of the paper is as follows. In Section \ref{sec_coin} we define the classical billiard and the coin billiard. Section \ref{sec_gen} contains a construction of a generating function of the coin billiard. In Section \ref{sec_kam} we prove Theorem \ref{theorem_a}, whereas Theorem \ref{theorem_b} is proved in Section \ref{sec_standardmapmather}. Section \ref{sec_birkhoff} contains the proof of Theorem \ref{theorem_c}. The results of the numerical experiments with the elliptical coins can be found in Section \ref{sec_num}. Finally, in the appendix, we describe the version of Moser's theorem which we use to prove Theorem \ref{theorem_a}.

\section{The Coin Mapping}\label{sec_coin}

In this section we define the classical billiard map, and subsequently the coin map. We describe several salient features of the latter. 

\subsection{The Classical Billiard}

Denote by $\G \subset \R^2$ a bounded strictly convex domain with $C^r$ ($r=2,3, \ldots, \infty, \omega$) boundary. Rescaling $\R^2$ if necessary, we may assume without loss of generality that the length of $\partial \G$ is $2 \pi$, and introduce an arclength parametrisation $\gamma : \T \to \R^2$, where $\T := \R / 2 \pi \Z$. We assume that $\gamma$ is oriented counterclockwise. Denote by $\f \in \T$ the arclength parameter, and let $\A := \T \times (0, \pi)$. The billiard map $T_1: \A \to \A$ is defined by $T_1(\f, \t) = (\bar \f, \bar \t)$ in the following way. A phase point $(\f, \t)$ uniquely determines a point $x \in \partial \G$ and an inward-pointing unit velocity vector $v \in T_x^1 \R^2$ by $x = \gamma(\f)$ and $v = \mathcal R (\t) \gamma'(\f)$, where $\mathcal R (\t)$ is the usual counterclockwise rotation matrix by an angle $\t$. The billiard dynamics follows the straight line in the direction of $v$ until the next point of collision with the boundary at $\gamma (\bar \f)$, where the velocity vector $v$ is reflected according to the optical law: the angles of incidence and reflection coincide. Then $\bar \t$ is this new incidence/reflection angle. 

It is well-known that the billiard map $T_1$ is a $C^{r-1}$ twist diffeomorphism. Indeed, we lose one degree of differentiability in the finite regularity case because the definition of the map requires $\gamma'$. Moreover invertibility, as well as the regularity of the inverse follows from the easily verifiable fact that $\mathcal I \circ T_1 \circ \mathcal I = T_1^{-1}$, where $\mathcal I (\f, \t) = (\f, \pi - \t)$. The latter property is called \emph{reversibility}. The twist property $\partial \bar \f/\partial \t >0$ too is easily verifiable by elementary geometrical considerations: if we fix a phase point $(\f,\t) \in \A$, keep $\f$ unchanged, and consider $\t' > \t$, then the point $\gamma \circ \Pi_{\f} \circ T_1 (\f, \t')$ lies strictly in the positive counterclockwise direction from $\gamma \circ \Pi_{\f} \circ T_1 (\f, \t)$. 

Although the billiard map admits a pleasant geometrical definition, there is in general no closed formula for $T_1$. However, near the boundary $\partial \A$ we can compute its Taylor series expansion. Indeed, for values of $\t$ close to zero the expansion
\begin{equation}\label{eq_billexp}
	\bar \f = \f + 2 \,  \rho (\f)\, \t + O\left(\t^2 \right), \quad \bar \t = \t - \frac{2}{3} \, \rho'(\f)\, \t^2 + O\left(\t^3 \right)
\end{equation}
with $(\bar \f, \bar \t) = T_1(\f,\t)$ was computed in \cite{lazutkin1973existence}, where $\rho (\f)$ is the radius of curvature of $\partial \G$ at $\gamma (\f)$. Obviously an analogous expansion holds near the other part of the boundary, for values of $\t$ close to $\pi$, by the reversibility property. In the paper \cite{lazutkin1973existence} the author makes a further transformation to a remarkable system of coordinates (now commonly called \emph{Lazutkin coordinates}) in which the map $T_1$ admits an even simpler expansion. We note that Lazutkin coordinates are not helpful in the present paper because in those coordinates when we consider, as in the next section, the composition $T_2 \circ T_1$, we lose track of the near-integrable nature of the problem.

\subsection{The Coin Billiard}

Maintaining the notation of the previous section, and with $\ell >0$, we denote by $\C(\G, \ell)$ the cylinder of height $\ell$ with homeomorphic copies of $\G$ on the top and on the bottom, and call it the \emph{coin of height $\ell$}. We consider the motion of a particle on the surface of $\C(\G, \ell)$. On the top of the coin, the motion of the particle is equivalent to the billiard on $\G$. It starts at a point $\g(\f)$ with its unit velocity vector $v$ making an angle $\t$ with the boundary $\partial \G$. It moves in a straight line in the direction of $v$ until it reaches $\partial \G$ again. When the particle reaches the boundary $\partial \G$ with angle of incidence $\bar \t$ at $\g(\bar \f)$ (as given by $T_1$), rather than reflect back inside, the velocity vector reflects over the edge of the cylinder so that it makes an angle $\bar \t$ with the upper boundary of the cylinder. The particle then follows the geodesic on the surface of the cylinder until it reaches the copy of $\G$ at the bottom of the coin. Here the motion is reflected back inside $\G$ with an equal angle $\bar \t$ (see Figure \ref{figure_coin}). In order to calculate the augmentation of the $\f$ coordinate caused by following the geodesic on the cylinder $\T \times [0, \ell]$, we consider the universal cover $\R \times [0, \ell]$ of the cylinder. On this strip, the geodesic becomes a straight line, making an angle $\bar \t$ with the upper boundary. Its projection onto the real line thus has length $\Delta$ which can be calculated from $\tan \bar \t = \frac{\ell}{\Delta}$. 

\begin{figure}[h]\label{coin}
	\centering
	\includegraphics[width=1\textwidth]{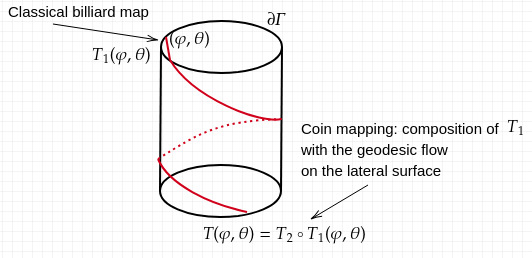}
	\caption{A schematic picture of the motion of the particle on the configuration coin.\label{figure_coin}}
\end{figure}

The map $T$ describing this motion is therefore given by $T=T_2 \circ T_1$ where $T_1$ is the billiard map from the previous section, and $T_2(\f, \t) = (\f + \ell \cot \t, \t)$. 

\begin{example*}
	As an example, consider the coin $\C \left( \mathbb D, \ell \right)$ where $\mathbb D = \{ q \in \R^2 : \| q \| = 1 \}$ is the unit disc in $\R^2$ and $\ell >0$. It is well-known (and can be proved by basic trigonometry) that the billiard map in $\mathbb D$ is given by $T_1(\f,\t) = (\f + 2 \t, \t)$. Therefore the coin map in $\C \left( \mathbb D, \ell \right)$ is given by $T(\f, \t) = (\f + 2 \t + \ell \cot \t,\t)$. It follows that every curve corresponding to constant values of $\t$ is invariant, so the map is completely integrable.
	
	Notice moreover that, writing $(\bar \f, \bar \t) = T(\f,\t)$, we have $\frac{\partial \bar \f}{\partial \t} = 2 - \frac{\ell}{\sin^2 \t}$. The latter has two zeros $\t_{\pm}(\ell) \in (0, \pi)$ whenever $\ell \in (0,2)$. It follows that there are two essential curves $\{ \t = \t_{\pm}(\ell) \}$ in $\A$ where the twist condition fails whenever $\ell \in (0,2)$, one such curve $\{ \t = \frac{\pi}{2} \}$ when $\ell = 2$, and whenever $\ell >2$, the coin map in $\C \left( \mathbb D, \ell \right)$ is a twist map. 
\end{example*}

More generally, for any billiard $\G$ if $\ell>0$ is small then there are curves where the twist condition fails, but if $\ell$ is sufficiently large then $T$ is a twist map. However for any $\ell$, the map $T$ is a twist map in a neighbourhood of the boundary $\partial \A$ of the phase space, for the following reason. Although we have no closed formula for $T$ in general (since we have none for $T_1$), we can combine our formula for $T_2$ with the expansion \eqref{eq_billexp} of $T_1$ to obtain the expansion
\begin{equation}\label{eq_coinexp}
	\bar \f = \f + \frac{\ell}{\bar \t} + O \left( \t \right) , \quad \bar \t = \t - \frac{2}{3} \, \rho' (\f) \, \t^2 + O \left( \t^3 \right)\ ,
\end{equation}
with $(\bar \f, \bar \t) = T(\f,\t)$. It is thus clear that $\frac{\partial \bar \f}{\partial \t} < 0$ for all sufficiently small values of $\t$. Analogous formulas hold for values of $\t$ close to $\pi$.

\section{Generating Function}\label{sec_gen}

In this section we construct a generating function for the coin map. Denote by $\lambda = \cos \t \, d \f$ the Liouville 1-form, so that the symplectic form is $\Omega = d \lambda = \sin \t \, d \f \wedge d \t$. Recall that a map $F:\A \to \A$ is exact symplectic if and only if $\oint_C \lambda = \oint_{F(C)} \lambda$ for any closed path $C \subset \A$, or equivalently, if there is a function $h:\A \to \mathbb R$ such that $F^* \lambda - \lambda = dh$. Sometimes, however, it is more convenient when $h$ is not a function of the phase space $\A$, but rather of the second power $\R^2$ of the universal cover $\R$ of the configuration space $\T$. We say that $h: \R^2 \to \R$ is a \emph{generating function} for (the lift to the universal cover of) $F$ if $F(\f,\t) = (\bar \f, \bar \t)$ if and only if $\partial_1 h(\f, \bar \f) = \cos \t$ and $\partial_2 h(\f, \bar \f) = - \cos \bar \t$. 

For the argument in Section \ref{sec_standardmapmather} we require the existence of a generating function for the coin map $T$. In the following lemma, we construct generating functions for $T_1$ and $T_2$. This implies immediately, of course, that $T_1$ and $T_2$ are exact symplectic, and in turn the coin map $T=T_2 \circ T_1$ is exact symplectic. However this \emph{does not imply} the existence of a global generating function for $T$, because the coin map $T$ is not necessarily a twist map, as pointed out in Section \ref{sec_coin}. Since the coin map does satisfy a twist property in a neighbourhood of $\partial \A$ a generating function can be shown to exist in those regions. For completeness, we include a proof of this fact. 

\begin{lemma}\label{lemma_generatingfunctions}
	The maps $T_1$, $T_2$, $T$ satisfy the following properties. 
	\begin{enumerate} 
		\item \label{item_genfunc1}
		The billiard map $T_1$ is exact symplectic with respect to the 1-form $\lambda$. Moreover the function $h_1(\f, \bar \f) = - \| \gamma (\f) - \gamma (\bar \f ) \|$ generates $T_1$ in the sense that $T_1 (\f,\t)=(\bar \f, \bar \t)$ if and only if $\partial_1 h_1(\f, \bar \f) = \cos \t$ and $\partial_2 h_1(\f, \bar \f) = - \cos \bar \t$. 
		\item \label{item_genfunc2}
		The map $T_2$ is exact symplectic with respect to the 1-form $\lambda$. Moreover the function $h_2(\f, \bar \f) = - \ell \sqrt{1 + \ell^{-2} \left( \bar \f - \f \right)^2}$ generates $T_2$ in the sense that $T_2 (\f,\t)=(\bar \f, \bar \t)$ if and only if $\partial_1 h_2(\f, \bar \f) = \cos \t$ and $\partial_2 h_2(\f, \bar \f) = - \cos \bar \t$. 
		\item \label{item_genfunc3}
		The coin map $T$ is exact symplectic with respect to the 1-form $\lambda$. Moreover there is a $\delta >0$ such that, defining the set $D = \left\{ (\f, \t) : \t \in (0, \delta) \right\}$, there is a function $h(\f, \bar \f)$ such that for $(\f,\t) \in D$, $T(\f, \t) = (\bar \f, \bar \t)$ if and only if $\partial_1 h(\f, \bar \f) = \cos \t$ and $\partial_2 h(\f, \bar \f) = - \cos \bar \t$. 
	\end{enumerate}
\end{lemma}

\begin{proof}
	Part \ref{item_genfunc1} of the lemma is well-known, and can easily be verified by direct differentiation of the function $h_1$. As for part \ref{item_genfunc2}, observe that the equation $\bar \f = \f + \ell \cot \t$ for the $\f$ component of $T_2$ can be rewritten to read 
	\begin{equation}\label{eq_invertcot1}
		\ell^{-1} ( \bar \f - \f) = \frac{\cos \t}{\sqrt{1 - \cos^2 \t}}
	\end{equation}
	so we obtain
	\[
	\cos \t = \frac{\ell^{-1} ( \bar \f - \f)}{\sqrt{1+\ell^{-2} ( \bar \f - \f)^2}}. 
	\]
	Therefore, differentiation of the function $h_2(\f, \bar \f) = - \ell \sqrt{1 + \ell^{-2} \left( \bar \f - \f \right)^2}$ implies that $\partial_1 h_2 \left( \f, \bar \f \right) = \cos \t$ and $\partial_2 h_2 \left( \f, \bar \f \right) = - \cos \t = - \cos \bar \t$. 
	
	We now address part \ref{item_genfunc3} of the lemma. The fact that $T=T_2 \circ T_1$ is exact symplectic follows immediately from the exact symplecticity of $T_1$, $T_2$. We fix $\delta >0$ small enough so that the expansion \eqref{eq_coinexp} makes sense, and $T$ satisfies the twist condition $\frac{\partial \bar \f}{\partial \t} <0$ in $D$. Define the set
	\[
	\Delta = \left\{ (\f_0, \f_1) \in \R^2 : \exists \, \t_0 \in (0, \delta) \textrm{ with } \Pi_{\f} \circ T (\f_0, \t_0) = \f_1 \right\}. 
	\]
	Let $(\f_0, \f_1) \in \Delta$. Then there is $\t_0 \in (0, \delta)$ with $T (\f_0, \t_0) = \f_1$. Observe that this $\t_0$ is unique, as a result of the twist property of $T$ in $D$. Indeed, if $\t_0' < \t_0$ then $\Pi_{\f} \circ T (\f_0, \t_0') > \f_1$, and we get the opposite inequality in the case where $\t_0' > \t_0$. It follows that there is a well-defined function $\Phi: \Delta \to \R$, with $\Phi(\f_0, \f_1) = \Pi_{\f} \circ T_1(\f_0, \t_0)$. 
	
	The map $\Phi$ is $C^{r-1}$ whenever the boundary is $C^r$. Indeed, this can be seen by direct computation of the derivatives, differentiating the equations
	\begin{align}
		\f_0 ={}& \Pi_{\f} \circ T_1^{-1} \left( \Phi (\f_0, \f_1), \arccos \partial_1 h_2 \left( \Phi (\f_0,\f_1),\f_1 \right) \right) \\
		\f_1 ={}& \Pi_{\f} \circ T_2 \left( \Phi (\f_0, \f_1), \arccos \left( - \partial_2 h_1 \left( \f_1, \Phi (\f_0, \f_1) \right) \right) \right).
	\end{align}
	We do not include the computations as they are lengthy, and as we do not require formulas for the derivatives of $\Phi$ in this paper. However it can be seen from these equations that the first derivatives of $\Phi$ depend on the first derivatives of $T_1$, $T_2$ and the second derivatives of $h_1$, $h_2$, so they are $C^{r-2}$, and so $\Phi$ is $C^{r-1}$. 
	
	We then define the function $h : \Delta \to \R$ by
	\[
	h(\f, \bar \f) = h_1 \left(\f, \Phi \left(\f, \bar \f \right) \right) + h_2 \left( \Phi \left(\f, \bar \f \right), \bar \f \right). 
	\]
	Suppose $(\f, \t) \in D$ and write $(\bar \f, \bar \t) = T ( \f, \t)$. Then, by the construction of $\Phi$, we have 
	\begin{align}
		\partial_1 h_1(\f, \Phi(\f, \bar \f)) = \cos \t, \quad \partial_2 h_1(\f, \Phi(\f, \bar \f)) = - \cos \bar \t \\
		\partial_1 h_2(\Phi(\f, \bar \f), \bar \f) = \cos \bar \t, \quad \partial_2 h_2(\Phi(\f, \bar \f), \bar \f) = - \cos \bar \t. 
	\end{align}
	Writing $\Phi = \Phi(\f, \bar \f)$, it follows that
	\[
	\partial_1 h(\f, \bar \f) = \partial_1 h_1 \left(\f, \Phi  \right) + \left( \partial_2 h_1 \left(\f, \Phi  \right) + \partial_1 h_2 \left( \Phi , \bar \f \right) \right) \partial_1 \Phi = \cos \t
	\]
	and
	\[
	\partial_2 h(\f, \bar \f) = \partial_2 h_2 \left( \Phi , \bar \f \right)  + \left( \partial_2 h_1 \left(\f, \Phi  \right) + \partial_1 h_2 \left( \Phi , \bar \f \right) \right) \partial_2 \Phi = - \cos \bar \t,
	\]
	so $h$ is the function we are looking for. 
\end{proof}

\section{Existence of KAM Curves}\label{sec_kam}

In this section we prove, in two different scenarios, the existence of some KAM curves. These are homotopically nontrivial invariant curves $C$ of the map $T$ in $\A$ such that $T|_C$ is conjugate to a rotation by a Diophantine frequency. The two scenarios we consider are: when the corresponding billiard table $\Gamma$ is close to a disc; and when the height $\ell$ of the coin is small. In each of these situations we make a coordinate transformation in a band in the phase space that is close to, but not accumulating on, the boundary. We then apply an existing KAM theorem (see Theorem \ref{theorem_moser} in the appendix) to establish the existence of the invariant curves. For simplicity of exposition, we only prove the existence of the invariant curves in the analytic category, but we point out that they their existence in less regular topologies would follow from existing KAM theorems that deal with those topologies. 

\subsection{The Near-Circular Case}

The main result of this section is the following.

\begin{theorem}\label{theorem_nckam}
Let $\Gamma_{\e} \subset \R^2$ be a real-analytic family of compact and strictly convex billiard domains, such that $\partial \Gamma_{\e}$ is a closed $C^{\omega}$ plane curve of length $2 \pi$, and such that $\Gamma_0$ is a disc. Fix a constant $c_1 >0$. Let $\ell>0$, and denote by $T$ the coin mapping with respect to $\C \left( \Gamma_{\e},\ell \right)$. Then there is an $\e_0 >0$ such that for any $\e < \e_0$ the subset $\T \times [\e - c_1 \, \e^2, \e + c_1 \, \e^2]$ of $\A$ contains $T$-invariant essential curves with Diophantine rotation numbers, and the union of these curves is a set of positive measure. 
\end{theorem}

\begin{proof}
In order to prove Theorem \ref{theorem_nckam} we use the assumptions to put the map $T$ in a form in which Theorem \ref{theorem_moser} can be applied. In a neighbourhood of the boundary, the coin map admits the expansion \eqref{eq_coinexp}. Fix some $c_1 >0$, as in the statement of Theorem \ref{theorem_nckam}. In the set $\T \times [\e - c_1 \, \e^2, \e + c_1 \, \e^2]$ we make the coordinate transformation $: (x,y) \mapsto (\f, \t)$ where 
\[
\f = x, \quad \t = \e - \e^2 y 
\]
so that $x \in \T$ and $y \in [-c_1, c_1]$. From \eqref{eq_coinexp} the map takes the form $(\bar x, \bar y) = T(x,y)$ with
\[
\bar x = x + \omega + \ell \bar y + O (  \e), \quad \bar y = y + \frac{2}{3} \rho'(x) + O (\e)
\]
where we have defined $\omega = \ell/\e \mod 2 \pi$.

Now, since $\partial \Gamma_0$ is a circle of radius $1$, we can expand the radius of curvature of $\Gamma_{\e}$ in powers of $\e$ to find that $\rho = 1 + \e \, \bar \rho + O \left( \e^2 \right)$. It follows that $\rho' = O(\e)$. Therefore the map $T$ becomes
\[
T:
\begin{dcases}
\bar x =& x + \omega + \ell \, y + O ( \e) \\
\bar y =& y +  O(\e).
\end{dcases}
\]
This map is real-analytic because, after making an analytic coordinate transformation, it is the composition of two analytic maps: $T_2$ is always analytic, independently of parameters, and $T_1$ is analytic because $\partial \Gamma_{\e}$ is. Obviously it is a twist map since $\frac{\partial \bar x}{\partial y} = \ell + O(\e) > 0$ for all sufficiently small values of $\e$. Finally, part \ref{item_genfunc3} of Lemma \ref{lemma_generatingfunctions} tells us that $T$ is exact symplectic (with respect to some Liouville 1-form) which in turn implies that $T$ satisfies the intersection property. Therefore Theorem \ref{theorem_moser} can be applied to complete the proof of Theorem \ref{theorem_nckam}. 
\end{proof}

\subsection{The Case when $\ell \ll 1$}

In this section we prove the following theorem. 

\begin{theorem}\label{theorem_shkam}
Let $\Gamma \subset \R^2$ be a compact and strictly convex billiard domain, such that $\partial \Gamma$ is a real-analytic closed plane curve of length $2 \pi$. Fix any $b_0 > a_0 >0$. Denote by $T$ the coin mapping with respect to $\C(\Gamma, \ell)$. Then there is an $\ell_0 >0$ such that for any $\ell \in (0, \ell_0)$ the subset $\T \times \left[ \frac{\ell}{b_0}, \frac{\ell}{a_0} \right]$ of $\A$ contains $T$-invariant essential curves with Diophantine rotation numbers, and the union of these curves is a set of positive measure. 
\end{theorem}

\begin{proof}
As in the proof of Theorem \ref{theorem_nckam}, we look for a suitable coordinate transformation that puts the map in a form where we can apply Theorem \ref{theorem_moser}. Fix $b_0 > a_0 >0$ as in the statement of the theorem, and choose real numbers $a, b$ such that $b>b_0>a_0>a >0$. In the set $\T \times \left[ \frac{\ell}{b}, \frac{\ell}{a} \right]$ we make the coordinate transformation $:(\f,\t) \mapsto (\xi, \eta)$ where 
\[
\xi = \f, \quad \eta = \frac{\ell}{\t}
\]
so that $\f \in \T$ and $\eta \in [a,b]$. Then from \eqref{eq_coinexp} we see that the map takes the form $T(\xi,\eta)=(\bar \xi, \bar \eta)$ where
\[
T:
\begin{dcases}
\bar \xi =& \xi + \eta + O (\ell) \\
\bar \eta =& \eta + O (\ell). 
\end{dcases}
\]
Then obviously the map is real-analytic, satisfies a twist condition for all sufficiently small $\ell >0$ since $\frac{\partial \xi}{\partial \eta} = 1 + O(\ell)$, and has the intersection property by part \ref{item_genfunc3} of Lemma \ref{lemma_generatingfunctions}. Therefore the rescaled map satisfies the assumptions of Theorem \ref{theorem_moser}, and so Moser's theorem completes the proof of Theorem \ref{theorem_shkam}. 
\end{proof}

\section{Construction of Vertically-Mapped Graphs}\label{auxiliary_section}

In this section we prove some statements that will be used later in the paper. In particular, we show that, due to its unbounded twist near the boundary, the coin map admits a sequence of smooth graphs accumulating on the boundary whose points are mapped onto the same abscissa after one iteration (for this reason, we sometimes refer to them as {\it vertically-mapped graphs}). This has important consequences because it determines a priori the regions where possible invariant curves may lie close to the boundary. The reasonings that we make here are loosely related to those made by Poincaré when he showed that, under a small perturbation, a curve with integer rotation number generically does not persist but nevertheless gives rise to points with integer rotation number \cite[Section 20]{arnoldavezergodicproblems}. Throughout this section, we will indicate by $\widetilde T$ the lift of $T$ to the universal cover $\R \times (0,\pi)$. The main result of this section is the following.

\begin{lemma}\label{grafici}
	There exists $m_0=m_0(\Gamma,\ell)\in \mathbb N$, and a sequence of disjoint well-ordered graphs $\{\mathscr G_m\}_{m\ge m_0}$ of $2\pi$-periodic functions $g_m\in C^4(\R)$ that satisfy the following properties:
	\begin{enumerate}
		\item  The graphs $\{\mathscr G_m\}_{m\ge m_0}$ accumulate uniformly on the boundary;
		\item For any $(\f,\t)\in \mathscr G_m$ one has $\Pi_\f \widetilde T(\f,\t)=\f+2m\pi$;
		\item  $|g_m'(\f)|\le \displaystyle\frac83 \rho''(\f) |g_m(\f)|^2$ for any $\f\in \T$.
	\end{enumerate}
\end{lemma} 
\begin{proof}
	
	Throughout this proof, for $m\in \mathbb N$, we use the notation
	\begin{align}\label{Fm}
		\begin{split}
			F_m(\f,\t):=&\Pi_\f(\widetilde T(\f,\t))-\f-2m\pi=\frac{\ell}{\t}+\frac{2}{3}\ell \rho'(\f)+ R(\f,\t)-2m\pi\ ,
		\end{split}
	\end{align} 
	where we have used the expansion \eqref{eq_coinexp} of $\widetilde T$ for small values of $\t$, and where $R(\f,\t)=O(\t)$.
	Then $F_m\in C^4$, as the billiard table is $C^5$ by hypothesis. Due to expression \eqref{Fm}, the twist property, and the continuity of $F_m$, for any fixed $\f_\star \in \T$ there exists $m_1=m_1(\Gamma,\ell,\f_\star )\in \mathbb N$, and a strictly decreasing sequence $\{\t_m=\t_m(\f_\star)\}_{m\ge m_1}$ such that one has  $\t_m\rightarrow 0$, and  $F_m(\f_\star ,\t_m)=0$. 
	
	Moreover, as $\t_m\rightarrow 0$, and $\partial_\t F_m(\f_\star ,\t_m)= -\ell/\t_m^2+O(1)$, for some positive integer $m_2=m_2(\Gamma,\ell,\f_\star )\ge m_1$ one has $\partial_\t F_m(\f_\star ,\t_m)\neq 0$ for any $m\ge m_2$. Then, for any given $m\ge m_2$ one can apply the Implicit Function Theorem, so that $F_m(\f,g_m(\f))=0$ in a neighborhood of $\f=\f_\star $, for an implicit function $g_m$ of class $C^4$ verifying $g_m(\f_\star )=\t_m$. 
	
	The same construction can be repeated around any other point $\f\in \T$, so that by the compactness of the domain, and by the uniqueness of the implicit function, there exists $m_0(\Gamma,\ell)$ such that for any given $m\ge m_0$ the function $g_m$ can be extended to the whole torus $\T$. Then, due to the $2\pi$-periodicity of $F_m(\f,\t)$ with respect to $\f$, one has that $g_m$ can be extended to the whole real line and has period $2\pi$. Moreover $g_{m+1}(\f)< g_m(\f)$ for any $\f\in \R$ and $m\ge m_0$, due to the negative twist property near the boundary.
	
	Then, as $g_m(\f)\rightarrow 0$ pointwise in $\R$, by Dini's Theorem and by the periodicity of $g_m$ one has the uniform convergence $\lim_{m\rightarrow + \infty}\sup_{x\in \R}|g_m(x)|=\lim_{m\rightarrow + \infty}\sup_{x\in \T}|g_m(x)|=0$, and the graphs $\mathscr{G}_m:=\text{graph}(g_m)$ accumulate uniformly on the horizontal axis.
	
	It remains to prove the estimate on the first derivative of $g_m$, which is a simple consequence of the Implicit Function Theorem. In fact, for $m$ sufficiently large, we have 
	\begin{align}
		\begin{split}
			|g_m'(\f)|=\left|\frac{\partial_\f F_m(\f,g_m(\f))}{\partial_\t F_m(\f,g_m(\f))}\right|\stackrel{\eqref{Fm}}{=}\left|\frac{2/3\,\ell\rho''(\f)+\partial_\f R(\f,\t)}{\ell/(g_m(\f))^2+O(1)}\right| \le \frac83 \rho''(\f) (g_m(\f))^2\ .
		\end{split}
	\end{align}
	This concludes the proof.

\end{proof}

Lemma \ref{grafici} has three important corollaries, that we state and prove below.

\begin{corollary}\label{width}
	Given $m\in \mathbb N$ sufficiently large, for any $\f\in \T$ one has 
	\begin{align}
		\begin{split}
			\max\{|g_{m-1}(\f)-g_{m}(\f))|,|g_{m}(\f)-g_{m+1}(\f))|\} \le \frac{3\pi}{\ell} (g_m(\f))^2\ .
		\end{split}
	\end{align}
\end{corollary}

\begin{proof}
	Fix $m\ge m_0$, $\f\in \T$, and write $\t_m=g_m(\f)$. With the notations of Lemma \ref{grafici}, we have $F_m(\f,\t_m)=0$. By the constructions of Lemma \ref{grafici}, and by the monotonicity of $F_m$ with respect to the second variable close to the boundary, it is sufficient to determine a real number $\alpha>0$ such that 
	$$
	\min\{|F_m(\f,\t_m+\alpha)|,|F_m(\f,\t_m-\alpha)|\}\ge 2\pi\ .
	$$ 
	More precisely, for $\alpha<\t_m$, by \eqref{Fm}, we can write 
	\begin{align}
			F_m&(\f,\t_m\pm\alpha)= \frac{\ell}{\t_m\pm \alpha}+\frac{2}{3}\ell \rho'(\f)+ R(\f,\t_m\pm \alpha)-2m\pi\\
			=& \frac{\ell}{\t_m}\mp\frac{\ell\alpha}{\t_m^2}\pm O\left(\ell\frac{\alpha^2}{\t_m^3}\right)+\frac{2}{3}\ell \rho'(\f)+ R(\f,\t_m)\pm\partial_\t R(\f,\widetilde \t_m)\alpha-2m\pi\\
			= & F_m(\f,\t_m)\mp\frac{\ell\alpha}{\t_m^2}\pm O\left(\ell\frac{\alpha^2}{\t_m^3} \right)\pm\partial_\t R(\f,\widetilde \t_m)\alpha\\
			=&\mp\frac{\ell\alpha}{\t_m^2}\pm O\left(\ell\frac{\alpha^2}{\t_m^3} \right)\pm\partial_\t R(\f,\widetilde \t_m)\alpha\ ,\label{dist}
	\end{align}
	where $\widetilde \t_m$ is a point on the line segment joining $\t$ to $\t\pm\alpha$, and where we have taken into account the fact that $F_m(\f,\t_m)=0$ by construction. It is clear from \eqref{dist} that, for $m$ sufficiently large, that is for sufficiently small $\t_m$, if one chooses $\alpha=\displaystyle \frac{3\pi}{\ell} \t_m^2$, then $\min\{|F_m(\f,\t_m+\alpha)|,|F_m(\f,\t_m-\alpha)|\}\ge 2\pi$, as required. 
\end{proof}

	\begin{corollary}\label{band}
	Let $C$ be an essential invariant graph of the map $\widetilde T$ having at least one point lying below the graph $\mathscr G_{2 m_0}$. Then, one has  $\mathscr G_m \cap C\neq \varnothing$ for at most one value of $m\ge 2 m_0$. In particular, there exists $m^\star=m^\star(C)\ge 2m_0$ such that $C$ is contained in the band having $\mathscr G_{m^\star-1}$ and $\mathscr G_{m^\star+1}$ as boundaries. 
\end{corollary}

\begin{proof}
For any $m\ge 2 m_0$, we define the sets  
\begin{align}\label{Em}
	\begin{split}
	\mathscr{D}_m:=&\{(\varphi,\theta)\in \mathscr G_m\ |\ \ \Pi_\t \widetilde  T(\varphi,\theta)\neq \theta\}\\
	\mathscr E_m:=&\{(\varphi,\theta)\in \mathscr G_m\ |\ \Pi_\t \widetilde  T(\varphi,\theta)=\theta \}= \mathscr G_m\setminus \mathscr D_m\ .
	\end{split} 
\end{align}
We firstly claim that $\mathscr D_m\cap C=\varnothing$ for any value of $m\ge 2m_0$. In fact, given $m\ge 2m_0$, if $\mathscr D_m=\varnothing$ there is nothing to prove. Otherwise, if $\mathscr D_m\neq \varnothing$ and $\mathscr D_m\cap C\neq \varnothing$ then the image $T(\f_\star,\t_\star)$ of any point $(\f_\star,\t_\star)\in \mathscr D_m\cap C$ would satisfy $\Pi_\t T(\f_\star,\t_\star)\neq \t_\star$ by definition of $\mathscr D_m$, and $T(\f_\star,\t_\star)\in C$ by invariance of $C$, but also $\Pi_\f T(\f_\star,\t_\star)=\f_\star$ by the second point of Lemma \ref{grafici}, in contradiction with the hypothesis that $C$ is the graph of a function. 

By the previous claim, for any $m\ge 2 m_0$ one has $\mathscr{G}_m\cap C=\mathscr{E}_m\cap C$. Then, if $\mathscr{E}_m\cap C\neq \varnothing$ for at most one value of $m\ge 2 m_0$, the proof ends here by taking into account the fact that the vertically-mapped graphs of Lemma \ref{grafici} are well-ordered. Now, assume that there exist two different values $m_3,m_4\ge 2m_0$, such that $\mathscr{E}_{m_3}\cap C\neq \varnothing$ and $\mathscr{E}_{m_4}\cap C\neq \varnothing$. By definition \eqref{Em}, for any $m\ge m_0$ points in $\mathscr E_m$ have rotation number $m$ under $\widetilde T$, so that points in $\mathscr{E}_{m_3}\cap C$ and in $\mathscr{E}_{m_4}\cap C$ have rotation number $m_3$ and $m_4\neq m_3$, respectively. However, $C$ is invariant under $\widetilde T$, so that its points must have a unique rotation number due to Poincaré's theorem, which leads to a contradiction. This concludes the proof.

\end{proof}

\begin{corollary}\label{sanremo}
	Given an integer $m\ge 2m_0$, and a point $\f\in \T$, consider any pair of points $(\f,\t_1)$ and $(\f,\t_2)$ in the band whose boundary is given by $\mathscr G_{m-1}$ and $\mathscr G_{m+1}$. Then one has 
	\begin{equation}
		|\Pi_\f\widetilde T(\f,\t_2)-\Pi_\f\widetilde T(\f,\t_1)|\le 4\pi\ .
	\end{equation}
\end{corollary}
\begin{proof}
	It follows immediately from the constructions of Lemma \ref{grafici}: by the twist condition the abscissa of all points below $\mathscr G_{m-1}$ (respectively, above $\mathscr G_{m+1}$) increases by a quantity larger than $2(m-1)\pi$ (respectively, smaller than $2(m+1)\pi$) under one iteration of $\widetilde T$. 
\end{proof}

A direct consequence of the previous corollaries is the following 

\begin{lemma}\label{lemma_herman}
	There is a constant $c_0=c_0(\Gamma,\ell)>0$ with the following property. If $C$ is the graph of an essential invariant curve of $T$ lying sufficiently close to the boundary then, setting $\t_0 = (\t_+ + \t_-)/2$ where $\t_{\pm}$ are the maximal and minimal values of $\t$ over all $(\f, \t) \in C$, one has $C \subset \T \times \left( \t_0-c_0 \t_0^2, \t_0 + c_0 \t_0^2\right)$. 
	
\end{lemma}

\begin{proof}
	Without any loss of generality, we can suppose that $C$ satisfies the hypotheses of Corollary \ref{band}. Then, calling $\xi: \T \rightarrow (0,\pi)$ the Lipschitz curve of which $C$ is the graph, and denoting by $m^\star=m^\star(\xi)$ the value defined in Corollary \ref{band}, we have that 
	\begin{equation}\label{ordinata_curva}
		|\xi(\f_2)-\xi(\f_1)|\le  |g_{m^\star-1}(\f_2)-g_{m^\star+1}(\f_1)| \qquad \forall \ \f_1,\f_2 \in \T,
	\end{equation}
	as $C$ must be contained in the band having as boundaries the graphs of $g_{m^\star+1}$ and $g_{m^\star-1}$. We estimate the right hand side of \eqref{ordinata_curva} by giving a bound to the terms $|g_{m^\star-1}(\f_1)-g_{m^\star+1}(\f_1)|$ and $|g_{m^\star-1}(\f_2)-g_{m^\star-1}(\f_1)|$ separately. We have 
	\begin{equation}\label{ho sete}
	|g_{m^\star-1}(\f_1)-g_{m^\star+1}(\f_1)| \le \frac{6\pi}{\ell} (g_{m^\star}(\f))^2	\le \frac{6\pi}{\ell} (g_{m^\star-1}(\f))^2
	\end{equation} 
	due to Corollary \ref{width} and to the fact that $g_m(\f)<g_{m+1}(\f)$ for any $\f\in \R$ and $m\ge m_0$ by Lemma \ref{grafici}. We also observe that, being $g_{m^\star-1}$ regular enough and $2\pi$-periodic, one has 
	\begin{align}\label{ho sonno}
		\begin{split}
		|g_{m^\star-1}(\f_2)-g_{m^\star-1}(\f_1)|\le & \max_{x\in \T}|g'_{m^\star-1}(x)|\ |\f_2-\f_1|\\
		 \le & \frac{16\pi}{3} \max_{x\in \T}\left\{|\rho''(x)|\,|g_{m^\star-1}(x)|^2\right\} \ ,
		\end{split}
	\end{align}
	where the last inequality comes from the third part of Lemma \ref{grafici}. Then, by choosing a point $x_\star$ at which $\max_{x\in \T}|g_{m^\star-1}(x)|=|g_{m^\star-1}(x_\star)|$, we have that, for any $0<\alpha<\xi(x_\star)$
	\begin{align}\label{var ascissa}
		\begin{split}
			&\Pi_\f \widetilde T(x_\star,\xi(x_\star)+\alpha)-\Pi_\f \widetilde T(x_\star,\xi(x_\star))\\
			&\stackrel{\eqref{eq_coinexp}}{=} \frac{\ell}{\xi(x_\star)+\alpha}+ R(x_\star,\xi(x_\star)+\alpha)-\frac{\ell}{\xi(x_\star)}- R(x_\star,\xi(x_\star))\\
			& = -\frac{\ell \alpha}{\xi(x_\star)^2}+O\left(\frac{\ell\alpha^2}{\xi(x_\star)^3}\right)+ \partial_\t R(x_\star, \widetilde \xi)\alpha \ ,
			\end{split}
	\end{align}
	where $\widetilde \xi$ is a point on the segment relying $\xi(x_\star)$ to $\xi(x_\star)+\alpha$. By choosing the value $\alpha= 5\pi \xi(x_\star)^2/\ell$ in formula \eqref{var ascissa}, we have that - if $C:=\text{graph}(\xi)$ is sufficiently close to the boundary - then
	\begin{equation}\label{girogirotondo}
		\Pi_\f \widetilde T(x_\star,\xi(x_\star)+\alpha)<\Pi_\f \widetilde T(x_\star,\xi(x_\star))- 4\pi\ .
	\end{equation}
	Taking Lemma \ref{sanremo} and the twist condition into account, estimate \eqref{girogirotondo} means that the point $\xi(x_\star)+\alpha$ lies above $\mathscr G_{m^\star-1}:=\text{graph}(g_{m^\star-1})$, so that 
	\begin{align}
		\begin{split}
		|g_{m^\star-1}(x_\star)|=\max_{x\in \T}|g_{m^\star-1}(x)|  < \xi(x_\star)+\alpha
		= \xi(x_\star)+\frac{5\pi}{\ell}\xi(x_\star)^2
		\le \t_+ +\frac{5\pi}{\ell}\t_+^2 \ ,
	\end{split}
	\end{align}
	and, as $\t_+=2(\t_0-\t_-)\le 2\t_0$, we finally obtain 
	\begin{equation}\label{ho fame}
		\max_{x\in \T}|g_{m^\star-1}(x)|\le 2 \t_0 +\frac{10\pi}{\ell}\t_0^2\ .
	\end{equation}
	The proof follows by plugging estimate \eqref{ho fame} into formulas \eqref{ho sete} and \eqref{ho sonno}, and by taking expression \eqref{ordinata_curva} into account. 
\end{proof}

\section{Drifting Trajectories when $\ell \gg 0$}\label{sec_standardmapmather}

The following theorem is the main result of this section. It is the precise statement of the result described by Theorem \ref{theorem_b}.

\begin{theorem}\label{theorem_nocurves}
Let $\Gamma \subset \R^2$ be a compact and strictly convex billiard domain, such that $\partial \Gamma$ is a closed $C^5$ plane curve of length $2 \pi$. Suppose that $\partial \Gamma$ is not a circle. Define 
\begin{equation}\label{eq_chaoscondition}
\ell_0 = - \frac{3}{\min_{\f \in \T} \rho''(\f)} >0. 
\end{equation}
Then for any $\ell > \ell_0$ and $\delta \ll \ell - \ell_0$ the coin mapping $T$ on the coin $\C \left(\G, \ell \right)$ has no essential invariant curves in the region $\T \times (0, \delta)$ of the phase space $\A$. 
\end{theorem}

The rest of this section is dedicated to the proof of this theorem. In fact, we give two proofs of the theorem, using different methods. The first proof is an adaptation of an argument of Mather using generating functions \cite{mather1982glancing}, while the second analyses the Lipschitz bounds of any hypothetical invariant essential curves. For simplicity, by an abuse of notation we will often denote by $T$ the lift of $T$ to the universal cover $\R \times (0, \pi)$ of $\A$; the distinction between the two will be made clear from the context. 

\begin{proof}[First proof of Theorem \ref{theorem_nocurves}]
Suppose now that $T$ has an essential invariant curve $C$ near the boundary. The form \eqref{eq_coinexp} of the map implies that $T|_{C}$ is orientation-preserving. 

Let $c_0>0$ be the constant from Lemma \ref{lemma_herman}, and let $\e = (\t_+ + \t_-)/2$ where $\t_{\pm}$ are the maximal and minimal values of $\t$ over all $(\f, \t) \in C$. In the strip $\T \times \left[ \e-c_0 \e^2, \e + c_0 \e^2 \right] \subset \A$ we make the coordinate transformation $: (x,y) \mapsto (\f, \t)$ where 
\[
\f = x, \quad \t = \e - \e^2 y 
\]
so that $x \in \T$ and $y \in [-c_0, c_0]$. From \eqref{eq_coinexp} the map takes the form $(\bar x, \bar y) = T(x,y)$ where
\[
\bar x = x + \omega + \ell \bar y + O (  \e), \quad \bar y = y + \frac{2}{3} \rho'(x) + O (\e)
\]
where we have defined $\omega = \displaystyle\frac{\ell}{\e} \mod 2 \pi$. If we remove the terms of order $\e$ we obtain the map
\[
T_0 : \begin{dcases}
\bar x = x + \omega + \ell \bar y \\
\bar y = y + \frac{2}{3} \rho'(x). 
\end{dcases}
\]
Rearranging these equations, we obtain $\bar y = \ell^{-1} (\bar x - x - \omega)$ and $y = \bar y - \frac{2}{3} \rho'(x)$. Therefore the function 
\[
h_0(x, \bar x) = \ell^{-1} \left(x \bar x - \frac{\bar x^2}{2} - \frac{x^2}{2} + \omega (\bar x - x) \right) - \frac{2}{3} \rho (x)
\]
which is $C^3$ because $\partial \Gamma$ is $C^5$, is such that $T_0(x,y)=(\bar x, \bar y)$ if and only if 
\begin{equation}\label{xxbar}
\partial_1 h_0 (x, \bar x) = y, \quad \partial_2 h_0 (x, \bar x) = - \bar y.
\end{equation} 
The second order partial derivatives of $h_0$ are
\begin{equation}\label{eq_h02ndderivs}
\partial_1^2 h_0(x,\bar x) = - \ell^{-1} - \frac{2}{3} \rho''(x), \quad \partial_2^2 h_0(x,\bar x) = - \ell^{-1}, \quad \partial_2 \partial_1 h_0(x,\bar x) = \ell^{-1}, 
\end{equation}
which we will use below. 

In what follows we want to use the function $h_0$ to describe trajectories of $T$ at first order, but $T$ is not exact symplectic with respect to the 1-form $y \, dx$, and so does not admit a generating function of this type. Lemma \ref{lemma_generatingfunctions}, however, implies that there is a function $h(x, \bar x)$ such that $T(x,y) = (\bar x, \bar y)$ if and only if
\begin{equation}\label{eq_cosgeneps}
\partial_1 h (x, \bar x) = \cos \left( \e - \e^2y \right), \quad \partial_2 h (x, \bar x) = -\cos \left( \e - \e^2 \bar y \right).
\end{equation}
Since $\cos \left( \e - \e^2y \right) = \left(1 - \frac{\e^2}{2} \right) + \e^3 y + O \left( \e^4 \right)$, and $y=\partial_1h_0(x,\bar x)$ at first order in $\e$ due to \eqref{xxbar}, it can be seen that the generating function $h$, up to additive constants, takes the form
\[
h ( x, \bar x) = \left( 1 - \frac{\e^2}{2} \right) (x - \bar x) + \e^3 \,  h_0 (x, \bar x) + O \left( \e^4 \right). 
\]

Now, we define the function 
\[
H(x_0, x_1, x_2) = \partial_2 h (x_0, x_1) + \partial_1 h (x_1, x_2). 
\]
Using \eqref{eq_h02ndderivs} we can compute the first derivatives of $H$:
\begin{align}
\frac{\partial H}{\partial x_0} ={}& \partial_1 \partial_2 h (x_0, x_1) = \e^3 \partial_1 \partial_2 h_0 (x_0, x_1) + O \left( \e^4 \right) = \e^3 \ell^{-1} + O \left( \e^4 \right) \\
\frac{\partial H}{\partial x_1} ={}& \partial_2^2 h(x_0,x_1) + \partial_1^2 h(x_1,x_2) = \e^3 \left[ \partial_2^2 h_0(x_0,x_1) + \partial_1^2 h_0(x_1,x_2) \right] + O \left( \e^4 \right) \\
={}& \e^3 \left[ -2 \ell^{-1} - \frac{2}{3} \rho''(x_1) \right] + O \left( \e^4 \right) \\
\frac{\partial H}{\partial x_2} ={}& \partial_2 \partial_1 h (x_1, x_2) = \e^3 \partial_2 \partial_1 h_0 (x_1, x_2) + O \left( \e^4 \right) = \e^3 \ell^{-1} + O \left( \e^4 \right) \ .
\end{align}

Recall we have assumed that there is an invariant essential curve $C$ of $T$, such that $C \subset 
\T \times [-c_0, c_0]$ in $(x,y)$ coordinates, and we have pointed out that $T|_C$ is necessarily orientation-preserving. Equation \eqref{eq_cosgeneps} implies that the short sequence $x_0, x_1, x_2$ represents a segment of a trajectory of $T$ if and only if $H(x_0,x_1, x_2)=0$. Since $\G$ is not a circle, there are both points on $\T$ for which $\rho'' >0$ and points for which $\rho''<0$. Choose $x_1 \in \T$ so that $x_1$ minimises $\rho''$. In particular, $\rho''(x_1) <0$. Since $C$ is an essential curve there exists $y_1 \in (-c_0, c_0)$ such that $(x_1, y_1) \in C$. Denote by $x_0$ the $x$-component of $T^{-1}(x_1, y_1)$, and by $x_2$ the $x$-component of $T(x_1, y_1)$. Then $x_0, x_1, x_2$ represents a segment of a trajectory of $T$, so $H(x_0, x_1, x_2)=0$. Moreover, since $\ell > - \frac{3}{\rho''(x_1)}$ (indeed, see \eqref{eq_chaoscondition}) and since $\e >0$ is sufficiently small, we have $\frac{\partial H}{\partial x_1}(x_0, x_1, x_2) <0$. Therefore, by the implicit function theorem, there is a $C^1$ function $\nu(x_0',x_2')$ such that for all $(x_0',x_2')$ in a neighbourhood of $(x_0,x_2)$ we have $H(x_0',\nu(x_0',x_2'),x_2') = 0$. Moreover the derivatives of $\nu$ are given by
\[
\frac{\partial \nu}{\partial x_0'} (x_0,x_2) = - \frac{\frac{\partial H}{\partial x_0}}{\frac{\partial H}{\partial x_1}} < 0, \quad \frac{\partial \nu}{\partial x_2'} (x_0,x_2) = - \frac{\frac{\partial H}{\partial x_2}}{\frac{\partial H}{\partial x_1}} < 0. 
\]
It follows that $\nu$ is a decreasing function of $x_0, x_2$, and so $T|_C$ is orientation-reversing. This contradiction implies that there can be no essential invariant curves of the mapping in a neighbourhood of the boundary, as required. 
\end{proof}

\medskip

 We will now give a second proof of Theorem \ref{theorem_nocurves}. 
	For any $x\in \R$, hereafter we set $\|x\|_\T:=\inf_{p\in \Z}|x+2p\pi |$. The first intermediate result that we use is the following Lemma, which is inspired by Herman's ideas \cite[Proposition 2.2]{herman1983courbes}. 
\begin{lemma}\label{lemma_Lipschitz}
	If $C$ is the graph of an essential invariant curve $\xi$ of the map $T$ lying sufficiently close to the boundary then, setting $\t_0 = (\t_+ + \t_-)/2$ where $\t_{\pm}$ are the maximal and minimal values of $\t$ over all $(\f, \t) \in C$, the Lipschitz constant of $C$, denoted by 
	\begin{equation}
		\mathscr L_C:= \sup_{\substack{\f_0,\f_1\in \R\\ \f_0\neq \f_1}} \displaystyle \frac{|\xi(\f_1)-\xi( \f_0)|}{\|\f_1-\f_0\|_\T} = \sup_{\substack{\f_0,\f_1\in [0,2\pi)\\ 0<|\f_0- \f_1|\le \pi}} \displaystyle \frac{|\xi(\f_1)-\xi( \f_0)|}{|\f_1-\f_0|}
	\end{equation}
	is bounded by 
	\begin{equation}
		\mathscr L_C \le 2\left(\frac{1}{\ell}+\frac23\max_{x\in \T}\rho''(x)\right)\t_0^2\ .
	\end{equation}

\end{lemma}
We observe that, by Corollary \ref{band}, $\t_0\rightarrow 0$ as $C$ approaches the boundary, so that $\mathscr L_C\rightarrow 0 $ near the boundary as well.

\begin{proof}
	Choose two points $\f_0,\f_1 \in [0,2\pi)$ satisfying $0<|\f_1-\f_0|\le \pi $, that is $0<\|\f_1-\f_0\|_\T\le \pi$. We can suppose that $\f_1>\f_0$ without any loss of generality, and we define the quantities 
	$$
L=	L(\f_0,\f_1):=\displaystyle \frac{\xi(\f_1)-\xi( \f_0)}{\f_1- \f_0}=\frac{\xi(\f_1)-\xi( \f_0)}{\|\f_1-\f_0\|_\T}
	\ ,\ \ 
	N=N(\f_0,\f_1):=\left(
	\begin{matrix}
		1\\
		\displaystyle L
	\end{matrix}
	\right)
	.
	$$
	
	By setting 
	\begin{equation}\label{G}
	G(t):=\displaystyle \frac{\widetilde T(t\f_1+(1-t)\f_0,t \xi(\f_1)-(1-t)\xi(\f_0))}{\f_1-\f_0},
\end{equation}
	  we obtain the following expression for the dynamics restricted to $C=\text{graph}(\xi)$:
	\begin{align}\label{pushforward}
		\begin{split}
			&	\frac{\widetilde T(\f_1,\xi(\f_1))-\widetilde T(\f_0,\xi(\f_0))}{\f_1- \f_0}= G(1)-G(0)= D\widetilde T( \f_\star,\t_\star)N( \f_0,\f_1) ,
		\end{split}
	\end{align}
	where $(\f_\star,\t_\star)$ is a point on the segment joining $(\f_0,\xi(\f_0))$ to $(\f_1,\xi(\f_1))$, and where, taking the expansion of $\widetilde T$ near the boundary given by formula \eqref{eq_coinexp} into account, we have
	\begin{equation}\label{Jacobian}
		D\widetilde T( \f_\star,\t_\star)=	\left(
		\begin{matrix}
			1+\displaystyle \frac23 \ell \rho''( \f_\star)+O(\t_\star) & -\displaystyle\frac{\ell}{\t_\star^2}+O(1)\\
			-\displaystyle \frac23 \rho''(\f_\star)\t_\star^2+O(\t_\star^3) & 1-\displaystyle\frac23 \rho'(\f_\star) \t_\star +O(\t_\star^2)
		\end{matrix}
		\right).
	\end{equation}
	
	Similarly, the inverse dynamics reads
	
	\begin{align}\label{pushforwardinverse}
		\begin{split}
			\frac{\widetilde T^{-1}(\f_1,\xi(\f_1))-\widetilde T^{-1}(\f_0,\xi(\f_0))}{\f_1- \f_0}= & D(\widetilde T^{-1})( \f_\bullet,\t_\bullet)N( \f_0,\f_1),
		\end{split}
	\end{align}
	where, by denoting $(\widetilde \f, \widetilde \t)=\widetilde T^{-1}(\f_\bullet,\t_\bullet)$, we have 
	\begin{align}
			D(\widetilde T^{-1})( \f_\bullet,\t_\bullet)= &	(D\widetilde T)^{-1}(\widetilde \f,\widetilde \t)\\
			=&(1+O(\widetilde \t))\left(
			\begin{matrix}
				1-\displaystyle\frac23 \rho'(\widetilde \f) \widetilde \t +O(\widetilde \t^2) & \ \ \displaystyle\frac{\ell}{\widetilde \t^2}+O(1)\\
				\displaystyle \frac23 \rho''(\widetilde \f)\widetilde \t^2+O(\widetilde \t^3)  & 	1+\displaystyle \frac23 \ell \rho''( \widetilde \f)+O(\widetilde \t)
			\end{matrix}
			\right) \label{inverse_Jacobian}\\
	\end{align}
	where, in the second line, we have taken into account that $1/\det( DT(\widetilde \f,\widetilde \t))=(1+O(\widetilde \t))$. We observe that, by Lemma \ref{lemma_herman}, one has $\t_\star,\widetilde \t\in [\t_0-c_0\t_0^2,\t_0+c_0\t_0^2]$ for some constant $c_0=c_0(\Gamma,\ell)$. Also, one has $\t_0\rightarrow 0$ when $C$ approaches the boundary, due to Lemma \ref{grafici} and to Corollary \ref{band}, so that the expansions in \eqref{Jacobian} and in \eqref{inverse_Jacobian} are truly perturbative.

	In the case where $L\ge 0$, as we had chosen $\f_1>\f_0$ we see from  \eqref{pushforward} and from \eqref{Jacobian} that the condition in order for the dynamics to preserve orientation for $(\f_0,\xi(\f_0))$ and $(\f_1,\xi(\f_1))$ is that 
	\begin{align}
		\begin{split}
			&\frac{L\ell}{\t_\star^2}(1+O(\t_\star^2))<	1+\frac23\ell \rho''(\f_\star)+O(\t_\star)\Leftrightarrow \ \\ 
			& L<\left(\frac{1}{\ell}+\frac23 \rho''(\f_\star)+O(\t_\star)\right)\frac{\t_\star^2}{1+O(\t_\star^2)},
		\end{split}
	\end{align}
	which implies that for any choice of $\f_0$ and $\f_1$ one must necessarily have 
	\begin{equation}\label{boundL1}
		L\ge 0 \qquad \Longrightarrow \qquad 			 L\le 2\left(\frac{1}{\ell}+\frac23 \max_{x\in \T}\rho''(x)\right)\t_0^2.
	\end{equation}
	If $L<0$, then by formulas \eqref{pushforwardinverse} and \eqref{inverse_Jacobian} the dynamics is orientation preserving if and only if
	\begin{equation}
		-\frac{\ell L}{\widetilde \t^2}(1+O(\widetilde \t))<1+O(\widetilde \t) \ \Leftrightarrow\ |L|<\frac{1+O(\widetilde \t)}{1+O(\widetilde \t)}\frac{\widetilde \t^2}{\ell},
	\end{equation}
	which yields the necessary condition 
	\begin{equation}\label{boundL2}
		L<0\qquad \Longrightarrow \qquad 	|L|\le  \frac{2}{\ell}\t_0^2.
	\end{equation}
	As the bounds \eqref{boundL1} and \eqref{boundL2} are necessary conditions for the dynamics on $C$ not to be orientation reversing for any choice of the points $(\f_0,\xi(\f_0))$ and $(\f_1,\xi(\f_1))$, with $\f_1,\f_0\in[0,2\pi)$, $0<\| \f_1-\f_0\|_\T\le \pi$, and as $\max_\T\rho''(x)>0$ due to the periodicity of $\rho'(x)$, the bound on the global Lipschitz constant $\mathscr L_C$ given in the statement holds. 
\end{proof}

\begin{proof}[Second proof of Theorem \ref{theorem_nocurves}]
Since by hypothesis we have $\ell>\ell_0$, for any given height $\ell$ there exists a number $K=K(\ell,\ell_0)>1$ such that 
\begin{equation}\label{kappaellezero}
	\ell>-\frac{3K}{\min_{x\in \T}\rho''(x)}=K\ell_0.
\end{equation}
 We now adopt the setting and the notations of Lemmas \ref{lemma_herman} and \ref{lemma_Lipschitz}, and we fix a proper interval $[\f_0,\f_1]\subset \T$ such that
	\begin{equation}\label{controllo der seconda curvatura}
\rho''\left(	[\f_0,\f_1]\right)= 	I_{\Gamma,K}, \qquad \text{ where } I_{\Gamma,K}:=\left[\min_{x\in \T}\rho''(x),\min_{x\in \T}\rho''(x)/K\right]\ .
	\end{equation}
	We observe that such an interval $[\f_0,\f_1]$ exists because $\rho''$ is continuous ($\partial \Gamma$ is $C^5$), and that $I_{\Gamma,K}$ is not a singleton because the table $\Gamma$ is not a disc.

	Then, for $t\in [0,1]$, we  use the notations
	\begin{equation}\label{notazioni}
		\f_t:=t\f_1+(1-t)\f_0, \qquad \t_t:=t\xi(\f_1)+(1-t)\xi(\f_0)\ ,
	\end{equation}
	and
	we set $H(t):=\Pi_\f G(t):=\Pi_\f \widetilde T(\f_t,\t_t)/(\f_1-\f_0)$. One has $H\in C^1([0,1])$; by formulas \eqref{Jacobian}, \eqref{notazioni}, and by Lemma \ref{lemma_herman}, its derivative reads
	\begin{align}\label{accaprimo}
		\begin{split}
		H'(t)=&  	1+\frac23\ell \rho''(\f_t)+O(\t_t)-\frac{L\ell}{\t_t^2}(1+O(\t_t^2))\\
		= & 1+\frac23\ell \rho''(\f_t)-\frac{L\ell}{\t_0^2}+R(\f_t,\t_t)
		\end{split} 
	\end{align}
	where $L:=(\xi(\f_1)-\xi(\f_0))/(\f_1-\f_0)$, and where, in the last equality, for any $t\in [0,1]$ the function $R(\f_t,\t_t)$ is a remainder of order $O(\t_0)$ when $\t_0\rightarrow 0$ due to Lemma \ref{lemma_Lipschitz}. 
	Then, by \eqref{accaprimo}, one has
	\begin{equation}
	H'(t)<0 \quad \Longleftrightarrow \quad 	L>\left(\frac{1}{\ell}+\frac23 \rho''(\f_t)\right)\t_0^2+R(\f_t,\t_t)\frac{\t_0^2}{\ell},
	\end{equation}
	so that, due to \eqref{controllo der seconda curvatura} and to Lemma \ref{lemma_herman}, if the condition
	\begin{align}\label{minorante L}
		\begin{split}
	L>\left(\frac{1}{\ell}+\frac23\frac{ \min_{x\in \T} \rho''(x)}{K}\right)\t_0^2+R^\star ,\\
	 \text{ with }\quad 	R^\star:= \frac{\t_0^2}{\ell}\times \max_{\substack{\f\in [\f_0,\f_1] \\ \t\in[\t_0-c_0\t_0^2,\t_0+c_0\t_0^2] }} R(\f,\t),
	\end{split} 
	\end{align}
	is true, then $H'(t)<0$ for all $t\in [0,1]$. 
	We now suppose that condition \eqref{minorante L} holds; then, $H$ is strictly decreasing in $[0,1]$, and  one has 

	\begin{equation}\label{rev1}
		H(1)-H(0):=\frac{\Pi_\f\left[\widetilde  T(\f_1,\xi(\f_1))-\widetilde T(\f_0,\xi(\f_0))\right]}{\f_1-\f_0}<0.
	\end{equation} 
Expression \eqref{rev1}  indicates that $\widetilde T$ reverses the orientation of the points $(\f_1,\xi(\f_1))$ and $(\f_0,\xi(\f_0))$, in contradiction with the fact that, as we have already pointed out previously in the paper, $\widetilde T|_\xi$ preserves orientation. Therefore, condition \eqref{minorante L} must be false, and one must have 
\begin{equation}\label{maggiorante L}
	L\le \left(\frac{1}{\ell}+\frac23 \frac{\min_{x\in \T} \rho''(x)}{K}\right)\t_0^2+R^\star.
\end{equation}
Then, due to formulas \eqref{pushforwardinverse} and \eqref{inverse_Jacobian}, and to Lemma \ref{lemma_herman}, the inverse dynamics yields 

\begin{align}
		&\frac{\widetilde T^{-1}(\f_1,\xi(\f_1))-\widetilde T^{-1}(\f_0,\xi(\f_0))}{\f_1- \f_0}=(1+O(\widetilde \t))+\frac{\ell L}{\widetilde \t^2}(1+O(\widetilde \t)) \\
		&  \stackrel{\eqref{maggiorante L}}{\le } (1+O(\widetilde \t))+\frac{\ell }{\widetilde \t^2}(1+O(\widetilde \t))\left\{\left(\frac{1}{\ell}+\frac23 \frac{\min_{x\in \T} \rho''(x)}{K}\right)\t_0^2+R^\star\right\}\\
		& = (1+O(\t_0))+\frac{\ell }{ \t_0^2}(1+O( \t_0))\left\{\left(\frac{1}{\ell}+\frac23 \frac{\min_{x\in \T} \rho''(x)}{K}\right)\t_0^2+R^\star\right\}\\
		& = 2+\frac23 \ell \frac{\min_{x\in \T} \rho''(x)}{K}+\frac{\ell R^\star} {\t_0^2}+O(\t_0). \label{inverse reverse}
\end{align}
Now, we have that 
\begin{equation}
	2+\frac23 \ell \frac{\min_{x\in \T} \rho''(x)}{K}<0, \qquad \frac{\ell R^\star} {\t_0^2}=\max_{\substack{\f\in [\f_0,\f_1] \\ \t\in[\t_0-c_0\t_0^2,\t_0+c_0\t_0^2] }} R(\f,\t)=O(\t_0)
\end{equation}
due, respectively, to condition \eqref{kappaellezero}, and to the previous constructions. Moreover, as $K$ depends only on $\ell$ and $\ell_0$, the value of $2+\displaystyle \frac23 \ell \frac{\min_{x\in \T} \rho''(x)}{K}$ is a fixed quantity when $\t_0\rightarrow 0$. Hence, by expression \eqref{inverse reverse} we have that $\widetilde T^{-1}$ is orientation reversing for the points $(\f_1,\xi(\f_1))$ and $(\f_0,\xi(\f_0))$ when the graph of the curve $\xi$ lies too close to the boundary, which is once again impossible. 

We have derived a contradiction when either \eqref{minorante L} or \eqref{maggiorante L} is satisfied, hence there exists no essential invariant curve of $\widetilde T$ close to the boundary.

\end{proof}

\section{Extension of a Theorem of Bialy}\label{sec_birkhoff}

The goal of this section is to prove Theorem \ref{theorem_c}. As mentioned in the introduction, we prove that if the billiard table is noncircular then there exists a sequence of homotopically trivial regions (elliptic islands or chaotic regions) through which there passes no essential invariant curve of the coin map. In addition, we provide estimates on the Lebesgue measure of these regions, which is of order $O \left( \frac{1}{m^2} \right)$, where $m$ is the index of the closest vertically-mapped graph (see Section \ref{auxiliary_section}). The main result of this section is the following. 

\begin{theorem}\label{theorem_birkhoffzones}
Let $\Gamma \subset \R^2$ be a compact and strictly convex billiard domain, with $\partial \Gamma$ a closed $C^{5}$ plane curve of length $2 \pi$. Suppose, moreover, that $\Gamma$ is not a disc. Then
	\begin{enumerate}
		\item  For any $\ell>0$ there exists $M=M(\G,\ell)\in \mathbb N$, and a sequence $\{\mathscr I_m\}_{m\ge M}$ of open, bounded, connected, homotopically-trivial regions accumulating on $\partial \A$, and intersecting no essential invariant curve of the coin map $T$ on $\C(\G,\ell)$;
		\item The Lebesgue measure of these regions is
		\begin{equation}
			\text{meas} (\mathscr I_m)\ge \frac{\ell^2\,\max_{x\in \T}\rho'(x)}{8(3+2\ell \max_{x\in \T}\rho''(x))}\times \frac{ \text{meas} (J_\G)}{\left(2m \pi +1-\displaystyle \frac13 \ell \max_{x\in \T}\rho'(x)\right)^2}, 
		\end{equation}
		where $J_\G$ is the largest connected component of  the set $$
	\left	\{\f\in \T \ |\ \rho'(\f)\in [\max_{x\in \T}\rho'(x)/2,\max_{x\in \T}\rho'(x)]\right \};
		$$
		\item If $T$ is integrable, for each $m\ge M$ the boundary of $\mathscr I_m$ is given by pieces of essential $T$-invariant graphs of rotation number $m$. 
		
	\end{enumerate}

\end{theorem}

We now consider the setting and the notations of Section \ref{auxiliary_section}. The following simple result also turns out to be fundamental in order to prove Theorem \ref{theorem_birkhoffzones}. 

\begin{lemma}{\label{lemma_vertically_mapped_points}}
Given  $m\ge 2 m_0$, with $m_0$ as in Lemma \ref{grafici}, choose, if it exists, a point $(\f_\star ,\t_\star )\in \mathscr{G}_m$ such that $\Pi_\t \widetilde T(\f_\star ,\t_\star )\neq \t_\star $. Then there is a neighborhood of $(\f_\star ,\t_\star )$ that intersects no essential invariant curve of $\widetilde T$.
\end{lemma}

\begin{proof}
In the proof of Corollary \ref{band} we had defined the set \begin{align}\label{Dm}
	\begin{split}
		\mathscr{D}_m:=&\{(\varphi,\theta)\in \mathscr G_m\ |\ \ \Pi_\t \widetilde  T(\varphi,\theta)\neq \theta\}\ 
	\end{split} 
\end{align}
and we showed that, for any $m\ge m_0$, any possible invariant graph $C$ of $\widetilde T$ satisfies $C\cap \mathscr D_m=\varnothing$.  

Now, given $m\ge 2m_0$ and $(\f_\star ,\t_\star )\in \mathscr D_m$, suppose for contradiction that there exists a sequence of essential invariant curves $\{\xi_k\}_{k\in \mathbb N}$ whose graphs $\{C_k\}_{k\in \mathbb N}$  approach $(\f_\star ,\t_\star )$, that is 
$ \inf_{(\f,\t)\in C_k} |(\f,\t)-(\f_\star,\t_\star)|\rightarrow 0$. Then, as the graphs have uniformly bounded Lipschitz constants by Lemma \ref{lemma_Lipschitz}, one must have $(\f_\star ,\xi_k(\f_\star ))\rightarrow(\f_\star ,\t_\star )$. For any given $k$, the Lipschitz constant $\mathscr L_{C_k}$ of the graph $C_k$ must satisfy
\begin{equation}
\mathscr L_{C_k}\ge \frac{|\Pi_\t \widetilde T(\f_\star ,\xi_k(\f_\star ))-\xi_k(\f_\star )|}{\|\Pi_\f \widetilde T(\f_\star ,\xi_k(\f_\star ))-\f_\star \|_\T}.
\end{equation}
However, we have that $|\Pi_\t \widetilde  T(\f_\star ,\xi_k(\f_\star ))-\xi_k(\f_\star )|\rightarrow|\Pi_\t \widetilde  T(\f_\star ,\t_\star )-\t_\star |\neq 0$ by hypothesis,  and that $\|\Pi_\f \widetilde  T(\f_\star ,\xi_k(\f_\star ))-\f_\star \|_\T\rightarrow\|\Pi_\f \widetilde  T(\f_\star ,\t_\star )-\f_\star \|_\T= 0$ by Lemma \ref{grafici} as $(\f_\star,\t_\star)\in \mathscr G_m$, so that the $\mathscr L_{C_k}$ are unbounded, in contradiction with Lemma \ref{lemma_Lipschitz}.
\end{proof}

We are now able to prove Theorem \ref{theorem_birkhoffzones}.

\begin{figure}[h]
	\centering
	\includegraphics[width=1\textwidth]{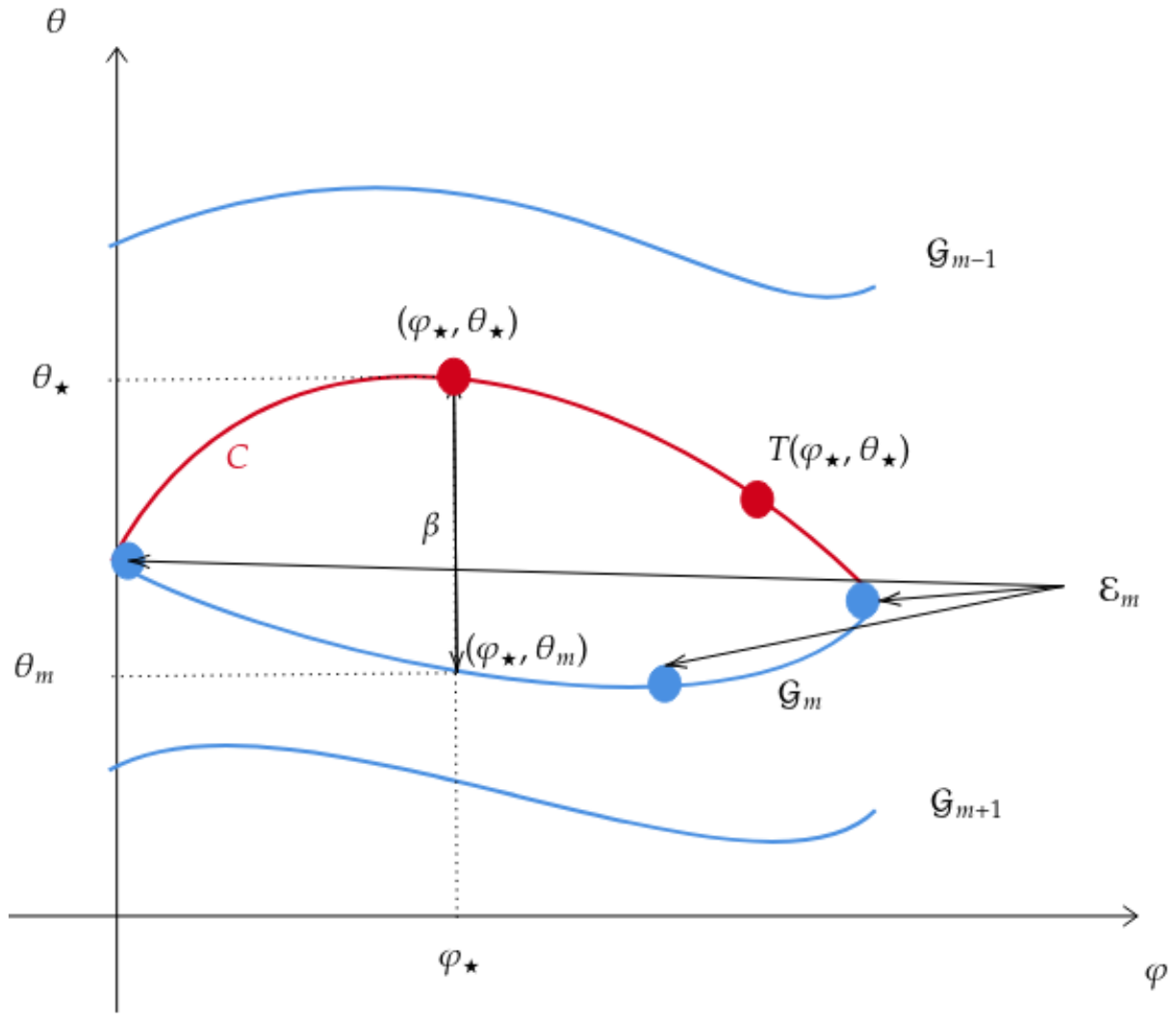}
	\label{figure_theoremc}
	\caption{A depiction of the proof of Theorem \ref{theorem_birkhoffzones}.}
\end{figure}

\begin{proof}[Proof of Theorem \ref{theorem_birkhoffzones}]
We proceed by steps. The first point of the thesis is proved at Step 1, whereas Step 2 contains the proof of the third point. Finally, the second point is proved at Steps 3-4.   
\medskip 

{\it Step 1. } For any sufficiently large integer $m$, we claim that as $\Gamma$ is not a disc one has
\begin{equation}
	\mathscr{D}_m:=\{(\varphi,\theta)\in \mathscr G_m\ |\ \ \Pi_\t \widetilde  T(\varphi,\theta)\neq \theta\}\neq \varnothing\ .
\end{equation}
In fact, if for contradiction $\mathscr D_m=\varnothing$, then each graph $\mathscr G_m$ would be flat and $\widetilde T$ would admit a sequence of lines of constant $\t$ approaching the boundary. As the dynamics on the variable $\t$ is the same for both the coin and the classical billiard map on the same table $\Gamma$, this implies that the classical billiard on $\Gamma$ would admit a sequence of circular caustics approaching $\partial \Gamma$, which means that $\partial \Gamma$ is a circle in contradiction with the hypothesis.

Then, by Lemma \ref{lemma_vertically_mapped_points}, each point in $\mathscr{D}_m$ has a neighborhood intersecting no essential invariant curve of $\widetilde T$. Taking the union over the points in $\mathscr{D}_m$ of the largest possible neighborhoods of this kind, one gets an open region $\mathscr I_m$ where no essential $\widetilde T$-invariant curve passes. We also observe that
\begin{equation}\label{Em}
	\mathscr E_m:=\{(\varphi,\theta)\in \mathscr G_m\ |\ \Pi_\t \widetilde  T(\varphi,\theta)=\theta \}= \mathscr G_m\setminus \mathscr D_m\neq \varnothing,
\end{equation}
as, due to part 3 of Lemma \ref{lemma_generatingfunctions}, $\widetilde T$ has the intersection property, so that, by construction, $\mathscr I_m$ is homotopically trivial.

{\it Step 2.} Fix $m$ sufficiently large; if any point $(\varphi,\theta)\in \mathscr E_m$ admits a neighborhood where no essential $\widetilde T$-invariant graph passes, then by construction there exists a whole tubular neighborhood of $\mathscr G_m$ having empty intersection with the essential invariant curves of $\widetilde T$, i.e. a so-called Birkhoff zone of instability. By classical results on twist maps of the annulus (see e.g. \cite{ angenententropy,katokhorseshoe, lecalvez1987proprietes, marcoentropy}) this implies that $\widetilde T$ is nonintegrable. Therefore, if $\widetilde T$ is integrable, there must exists a point $(\varphi,\theta)\in \mathscr E_m$ admitting a sequence of essential invariant graphs $\{C_k\}_{k\in \mathbb N}$ that approaches it, namely $\inf_{(\f',\t')\in C_k} |(\f',\t')-(\f,\t)|\rightarrow 0$. By Lemma \ref{lemma_Lipschitz}, the graphs $\{C_k\}$ are Lipschitz over $\T$, with a uniform bound on their constant. Then, the theorem of Ascoli and Arzelà ensures the existence of a subsequence $\{C_{k_j}\}_{j\in \mathbb N}$ converging uniformly to a Lipschitz graph $C$ having the same bound on the Lipschitz constant. As $\widetilde  T$ is continuous, and the graphs $C_k$ are all $\widetilde  T$-invariant by hypothesis, $C$ is also $\widetilde  T$-invariant. Moreover, by construction we have $(\varphi,\theta)\in C$. Now, the point $(\varphi,\theta)\in \mathscr E_m$ has rotation number $m$, as $\widetilde T(\f,\t)=(\f+2m\pi,\t)$, so that by Poincaré's Theorem the graph $C$ must have rotation number $m$ too. 

Now, it is known \cite[Proposition 2.17]{arnaudboundaries} that the boundary of each connected component of $\mathscr I_m$ is the union of pieces of two invariant continuous graphs having the same rational rotation number. Then, as $\mathscr E_m$ belongs to the boundary of $\mathscr I_m$ by construction, and as we have proved that any essential invariant graph passing through points in $\mathscr E_m$ must have rotation number equal to $m$, then the boundary of $\mathscr I_m$ is composed of pieces of invariant graphs of rotation number $m$.

{\it Step 3.} It remains to prove point 2 of the thesis, that is to estimate the Lebesgue measure of the regions $\mathscr I_m$, with $m$ large enough. If $\widetilde T$ admits no essential invariant curves near the boundary, there is nothing to prove, so we will make the assumption that $\widetilde T$ admits sequences of invariant graphs accumulating on the boundary. Then, choose an angle $\f_\star$ belonging to $J_\Gamma$, that is to the largest connected component of  the set $$
\left	\{\f\in \T \ |\ \rho'(\f)\in [\max_{x\in \T}\rho'(x)/2,\max_{x\in \T}\rho'(x)]\right \},
$$
and consider the unique angle $\t_m$ for which $(\varphi_\star,\theta_m)\in \mathscr G_m$. By formula \eqref{Fm}, we know that 
\begin{equation}\label{thetam}
	\begin{split}
		F_m(\f_\star,\t_m)=0 \quad \Longleftrightarrow \quad \t_m= \frac{\ell}{2m\pi-\displaystyle\frac23 \ell\rho'(\f_\star)-R(\f_\star,\t_m)}	
		\end{split}
\end{equation}
with $R(\f_\star,\t_m)=O(\t_m)$. Moreover, we also know by \eqref{eq_billexp} that, for any $(\f,\t)\in \T\times (0,\delta]$, with $\delta$ sufficiently small, the angle of reflection satisfies $\Pi_\t T(\f,\t) - \t=\bar \t-\t=-2/3 \rho'(\f)\t^2+O(\t^3)$, so that
\begin{equation}\label{curvatura_piccola}
	\Pi_\t T(\f,\t)= \t \qquad \Longleftrightarrow \qquad \frac23\rho'(\f)=O(\t).
\end{equation}
As $\Gamma$ is not a circle, the periodicity of $\rho(\f)$ implies $\rho'(\varphi_\star)\ge \max_{x\in\T}\rho'(x)/2>0 $, so that formula \eqref{curvatura_piccola} ensures that
$
\Pi_\t \widetilde  T(\varphi_\star,\theta_m) \neq \theta
$, 
that is $(\varphi_\star,\theta_m)\not \in \mathscr E_m$ for all sufficiently large $m$. 

Now, we know that the boundary of $\mathscr I_m$ is given by pieces of invariant graphs with rotation number $m$. Call $C$ the graph whose point of abscissa $\f_\star$ has minimal vertical distance from $(\f_\star,\t_m)$ ($C$ is not necessarily unique). We know that such a distance cannot be zero, due to Lemma \ref{lemma_vertically_mapped_points}, as we have just proved that $(\varphi_\star,\theta_m)\not \in \mathscr E_m$. So, we denote by $\theta_\star$ the unique angle satisfying $(\varphi_\star,\theta_\star)\in C$, and we set 
\begin{equation}
\beta=\beta(\f_\star):=\theta_\star-\theta_m\neq 0.
\end{equation}
We know that 
\begin{equation}\label{beta}
|\beta|\le \displaystyle \frac{3\pi}{\ell}\theta_m^2
\end{equation}
by Corollaries \ref{width} and \ref{band}.

{\it Step 4.} As $C$ is Lipschitz, the slope
\begin{equation}\label{Lower Lipschitz}
\mathscr L_C(\varphi_\star,\theta_\star):= \frac{|\Pi_\theta \widetilde T(\varphi_\star, \theta_\star)-\theta_\star|}{\|\Pi_\f \widetilde T(\varphi_\star, \theta_\star)-\varphi_\star\|_\T}
\end{equation} 
is well-defined and, as $(\f_\star,\t_\star)\in \mathscr D_m$ the denominator in the formula above is nonzero. 
We now estimate $\mathscr L_c(\f_\star,\t_\star)$. Sufficiently close to the boundary, hence for sufficiently small $\theta_m$, we have that (remember the bound \eqref{beta})
\begin{align}
|\Pi_\theta \widetilde T(\varphi_\star, \theta_\star)-\theta_\star|=&|\Pi_\theta \widetilde T(\varphi_\star, \theta_m+\beta)-\theta_m-\beta|\\
\stackrel{\eqref{eq_billexp}}{=}&\left|\frac23 \rho'(\varphi_\star)(\theta_m+\beta)^2+O(\theta_m+\beta)^3\right|\\
=&\left|\frac23 \rho'(\varphi_\star)\theta_m^2\left(1+O\left(\frac{\beta}{\theta_m}\right)\right)\right|\ge \frac12 \rho'(\varphi_\star)\theta_m^2 \label{bound theta}
\end{align}
and, by the computation in formula \eqref{dist},
\begin{equation}\label{bound phi}
\|\Pi_\f \widetilde T(\varphi_\star, \theta_\star)-\varphi_\star\|_\T=\left\|\-\frac{\ell\beta}{\theta_m^2}+O\left(\frac{\ell\beta^2}{\theta_m^3}\right)+\partial_\t R(\varphi_\star, \widetilde \theta_m)\beta+2m\pi\right\|_\T\le \frac{3\,\ell |\beta|}{2\,\theta_m^2},
\end{equation}
where $\widetilde \theta_m$ lies between $\theta_\star$ and $\theta_m$.
Putting \eqref{Lower Lipschitz} together with \eqref{bound theta} and \eqref{bound phi} we obtain the following lower bound on $\mathscr L_C(\varphi_\star,\theta_\star)$:
\begin{equation}\label{lower_bound}
\mathscr L_C(\varphi_\star,\theta_\star)\ge \frac{\rho'(\varphi_\star) \,\theta_m^4}{3 \ell\, |\beta|}.
\end{equation}
However, we know by Lemmas \eqref{lemma_herman} and \ref{lemma_Lipschitz} that $\mathscr L_C(\varphi_\star,\theta_\star)$ should not be too big either, in particular one must have  
\begin{equation}\label{bound_lipschitz}
\mathscr L_C(\varphi_\star,\theta_\star)\le 4 \left(\frac{1}{\ell}+\frac23 \max_{x\in \T}\rho''(x)\right)\theta_m^2
\end{equation}
where we have used the fact that $\theta_m=\theta_0+O(\theta_0^2)$. 

By comparing \eqref{lower_bound} with \eqref{bound_lipschitz}, we obtain
\begin{equation}\label{intermedia}
|\beta(\f_\star)|\ge \frac{\rho'(\varphi_\star)}{4(3+2\ell\max_{x\in \T}\rho''(x))}\theta_m^2.
\end{equation}

The construction above can be repeated for any $\f_\star\in J_\G$. Hence, taking Step 1 into account, the Lebesgue measure of $\mathscr I_m$ is bonded from below by the area of the rectangle whose base has length $\text{meas}(J_\G)$ and whose height has length $\min_{x\in J_\Gamma} |\beta(x)|$, that is
\begin{equation}\label{misura quasi}
	\text{meas} (\mathscr I_m)\ge  \min_{x\in J_\Gamma} |\beta(x)| \times \text{meas }(J_\Gamma)\stackrel{\eqref{intermedia}}{=}\frac{\max_{x\in \T}\rho'(x)\ \theta_m^2}{8(3+2\ell\max_{x\in \T}\rho''(x))}\text{meas }(J_\Gamma),
\end{equation}
where we have taken into account the fact that $\min_{x\in J_\G}\rho'(\f_\star)=\max_{x\in \T}\rho'(x)/2$. 
Finally, formula \eqref{thetam} implies that, for $m$ large enough, we can write
\begin{equation}
	\t_m= \frac{\ell}{2m\pi-\displaystyle\frac23 \ell\rho'(\f_\star)-R(\f_\star,\t_m)}\ge\frac{\ell}{2m\pi-\displaystyle\frac23 \ell\rho'(\f_\star)+1}, 
\end{equation}
so that, by plugging the above expression into \eqref{misura quasi}, and by using again the fact that $\rho'(\f_\star)\ge  \max_{x\in \T}\rho'(x)/2$, we obtain 
\begin{equation}
	\text{meas} (\mathscr I_m)\ge \frac{\ell^2\,\max_{x\in \T}\rho'(x)}{8(3+2\ell \max_{x\in \T}\rho''(x))}\times \frac{ \text{meas} (J_\G)}{\left(2m \pi +1-\displaystyle \frac13 \ell \max_{x\in \T}\rho'(x)\right)^2}.
\end{equation} 

This concludes the proof.
\end{proof}

\section{Numerical Experiments with Elliptical Coins}\label{sec_num}

In this section we show the phase portraits of several elliptical coin billiards $\C \left(\mathcal E_{a,b}, \ell \right)$ where $\mathcal E_{a,b} = \left\{ (x_1, x_2) \in \R^2 : x_1^2/a^2 + x_2^2/b^2 = 1 \right\}$. The height of the coin is fixed at $\ell=1.3$, the semiminor axis of the ellipse is fixed at $b=1$, and the semimajor axis takes the values $a=1, 1.2, 1.4, 1.6, 1.8, 2, 3, 4$. The first four pictures are contained in Figure \ref{figure_exp1pt1}, and the other four in Figure \ref{figure_exp1pt2}. 

Obviously for $a=1$ the ellipse is simply a circle, and so the coin mapping is completely integrable, with the angle of incidence/reflection $\t$ being a constant of motion. This can be seen clearly in Figure \eqref{figure_exp1pt1a} where the dynamics takes place on horizontal lines in the phase space. When we increase the semimajor axis $a$ and leave $b$ fixed, we obtain noncircular ellipses. With regards to the billiard in a noncircular ellipse, it is well-known that there are exactly two 2-periodic orbits: one along the major axis and one along the minor axis. Moreover the 2-periodic orbit along the major axis is hyperbolic, whereas the orbit along the minor axis is elliptic. These 2-periodic orbits of $T_1$ remain 2-periodic orbits of $T$, because they correspond to $\t = \frac{\pi}{2}$, and $T_2 \left(\f, \frac{\pi}{2} \right) = \left(\f, \frac{\pi}{2} \right)$ for any value of $\f$. Furthermore it can be shown by direct computation of the eigenvalues of $DT$ at the 2-periodic orbit on the minor axis that this orbit remains an elliptic orbit for $T$. This can clearly be seen in each phase portrait other than \eqref{figure_exp1pt1a}, where there are large elliptic islands around the elliptic periodic points.  

In Figures \eqref{figure_exp1pt1b} and \eqref{figure_exp1pt1c} we can clearly see some invariant curves near the boundary of the phase space, as described by Theorem \ref{theorem_nckam}. However these invariant curves appear to be destroyed when the parameter $a$ is increased further. In addition some elliptic islands are visible near the boundary for each value of $a$ at least up to $1.8$. When $a$ is moved beyond 1.8, the chaotic sea becomes larger, destroying the elliptic island in the centre of the phase portrait \eqref{figure_exp1pt2a}. There still appears however to be some regularity near the boundary in \eqref{figure_exp1pt2b}. However in \eqref{figure_exp1pt2c} and \eqref{figure_exp1pt2d} no regularity can be seen near the boundary, in keeping with Theorem \ref{theorem_nocurves}, and the chaotic sea appears to reach all of the phase space except for the elliptic islands surrounding the 2-periodic elliptic points. 

\begin{figure} 
    \centering
    \begin{subfigure}[b]{0.45\textwidth}
        \includegraphics[width=\textwidth]{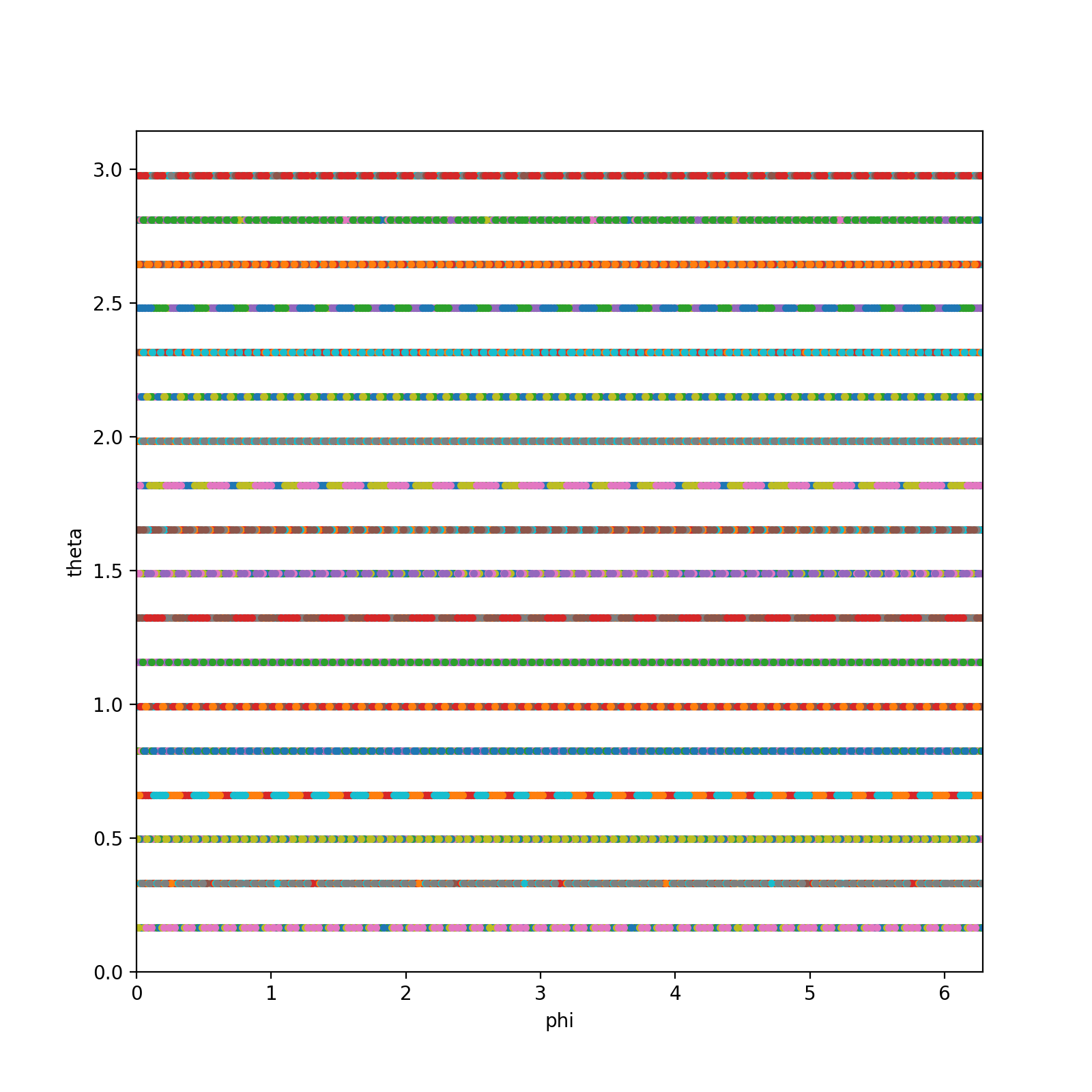}
        \caption{$a=1, \, b=1, \, \ell=1.3$}
        \label{figure_exp1pt1a}
    \end{subfigure} 
    \begin{subfigure}[b]{0.45\textwidth}
        \includegraphics[width=\textwidth]{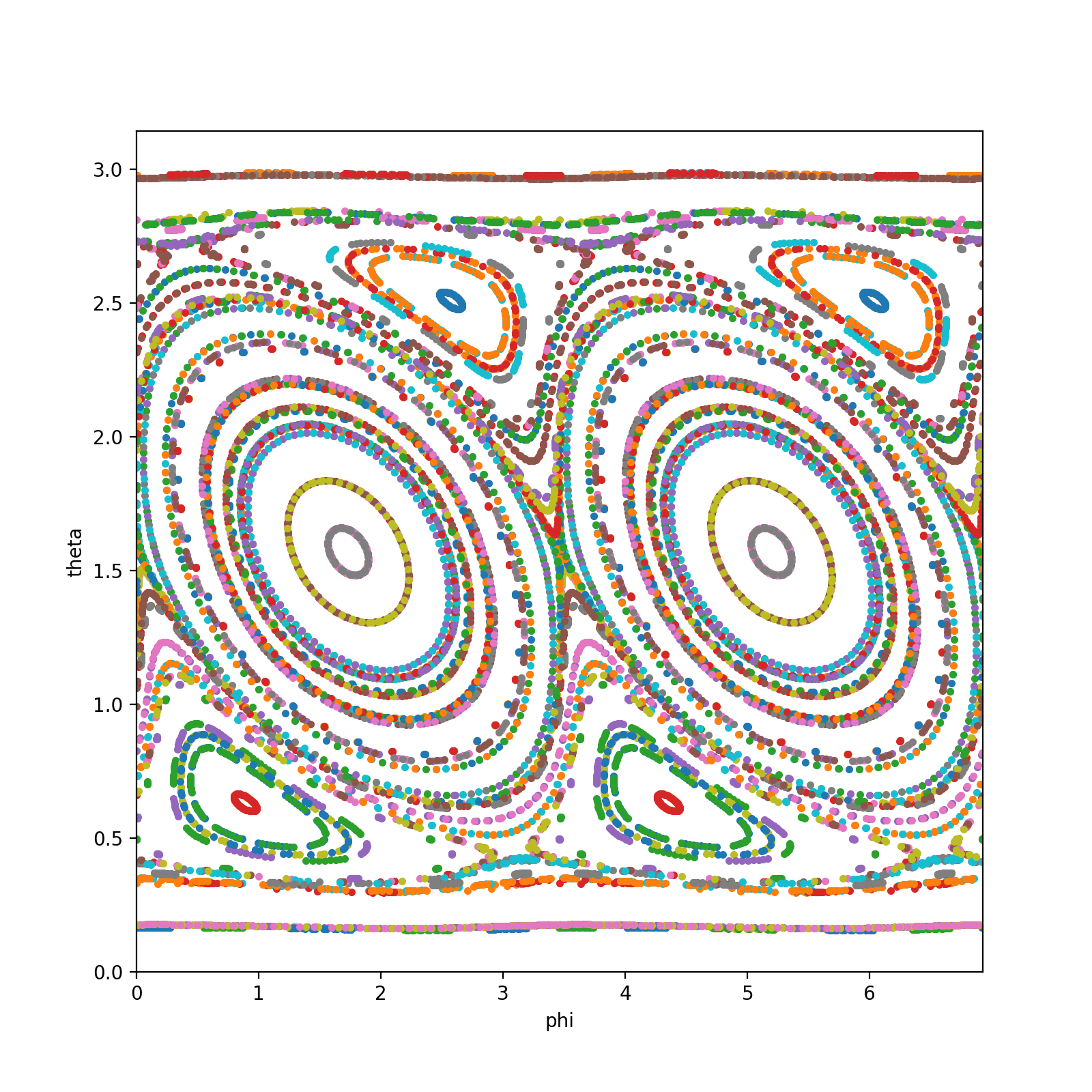}
        \caption{$a=1.2, \, b=1, \, \ell=1.3$}
        \label{figure_exp1pt1b}
    \end{subfigure}

    \begin{subfigure}[b]{0.45\textwidth}
        \includegraphics[width=\textwidth]{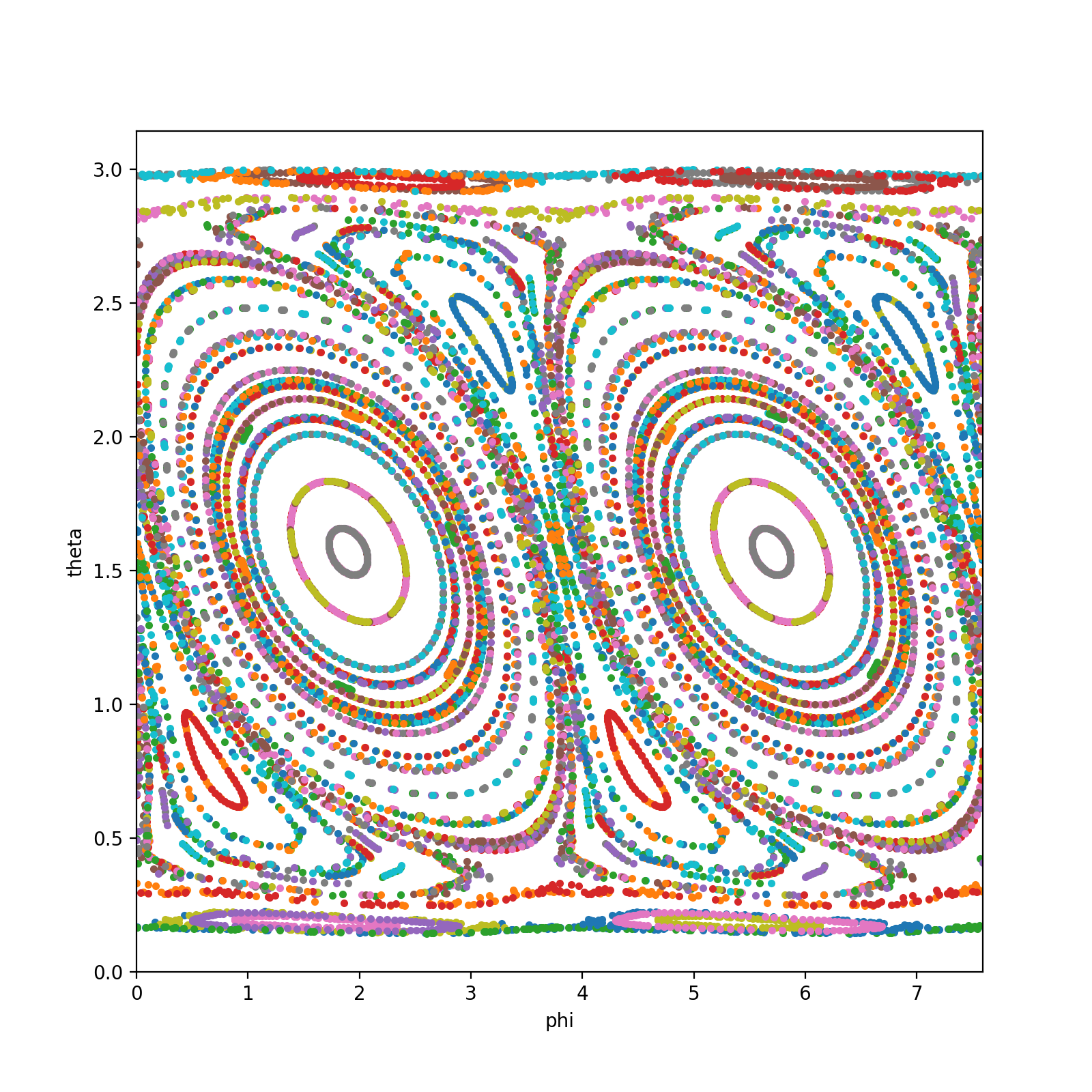}
        \caption{$a=1.4, \, b=1, \, \ell=1.3$}
        \label{figure_exp1pt1c}
    \end{subfigure}  
    \begin{subfigure}[b]{0.45\textwidth}
        \includegraphics[width=\textwidth]{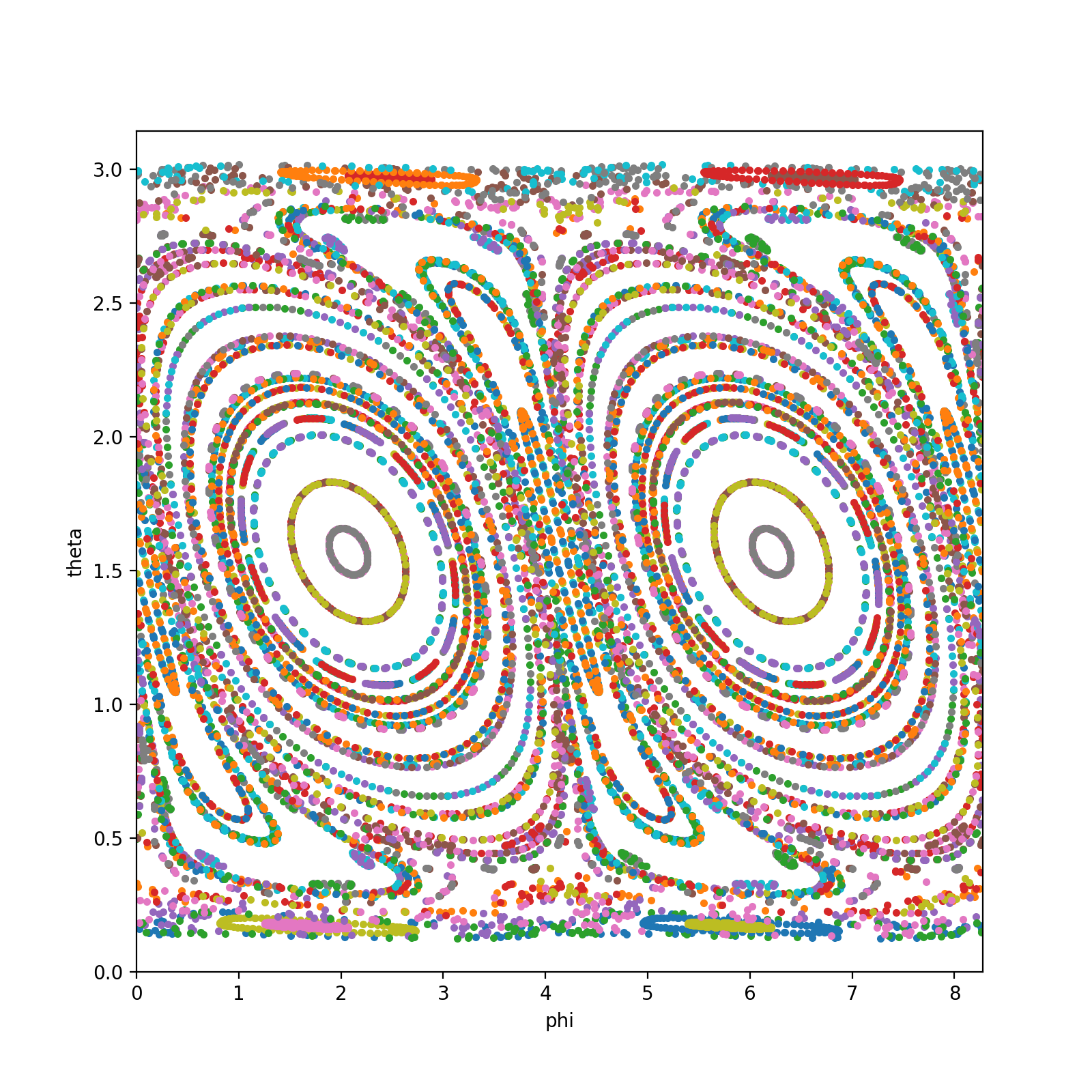}
        \caption{$a=1.6, \, b=1, \, \ell=1.3$}
        \label{figure_exp1pt1d}
    \end{subfigure}
    \caption{Each picture shows the first 100 iterations of various initial conditions under the coin mapping $T$ corresponding to the coin $\C \left( \mathcal{E}_{a,b},\ell \right)$, where $\mathcal{E}_{a,b}$ is ellipse with semimajor axis $a$ and semiminor axis $b$. We chose 144 different initial conditions spread evenly through the phase space. Different colours represent different orbits. The horizontal axis is the arclength parameter, and the vertical axis is the angle of incidence/reflection.}\label{figure_exp1pt1}
\end{figure}

\begin{figure} 
    \centering
    \begin{subfigure}[b]{0.45\textwidth}
        \includegraphics[width=\textwidth]{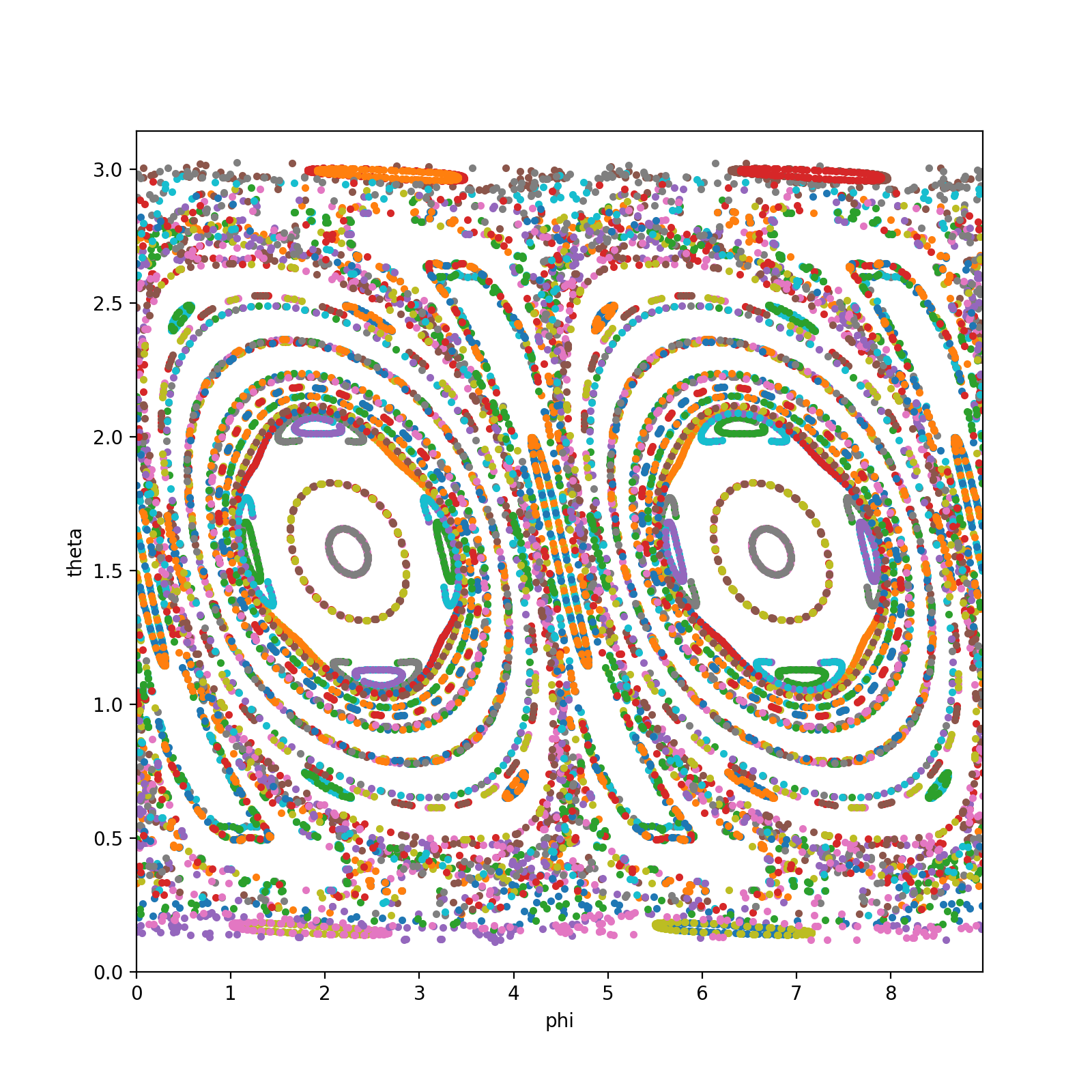}
        \caption{$a=1.8, \, b=1, \, \ell=1.3$}
        \label{figure_exp1pt2a}
    \end{subfigure} 
    \begin{subfigure}[b]{0.45\textwidth}
        \includegraphics[width=\textwidth]{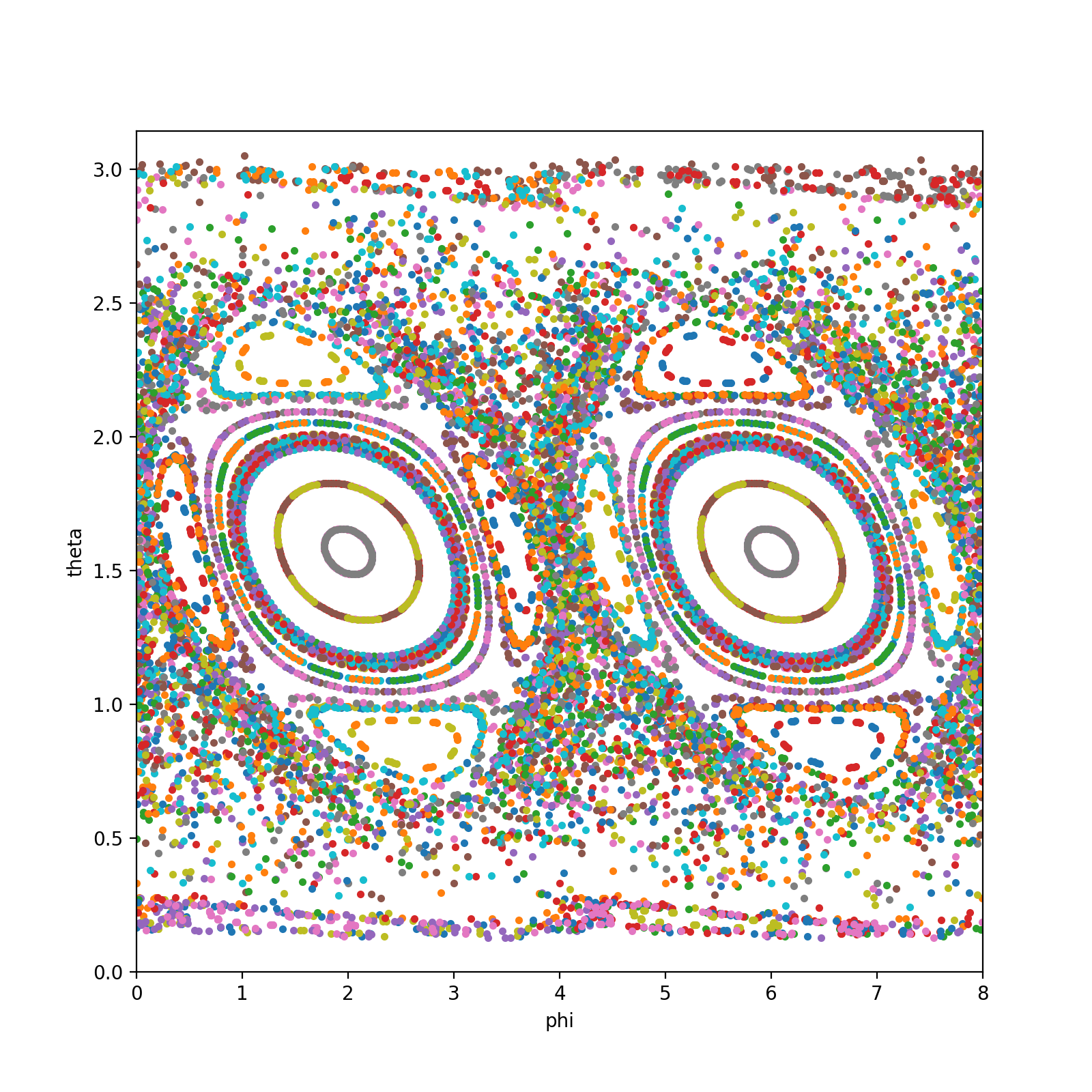}
        \caption{$a=2, \, b=1, \, \ell=1.3$}
        \label{figure_exp1pt2b}
    \end{subfigure}

    \begin{subfigure}[b]{0.45\textwidth}
        \includegraphics[width=\textwidth]{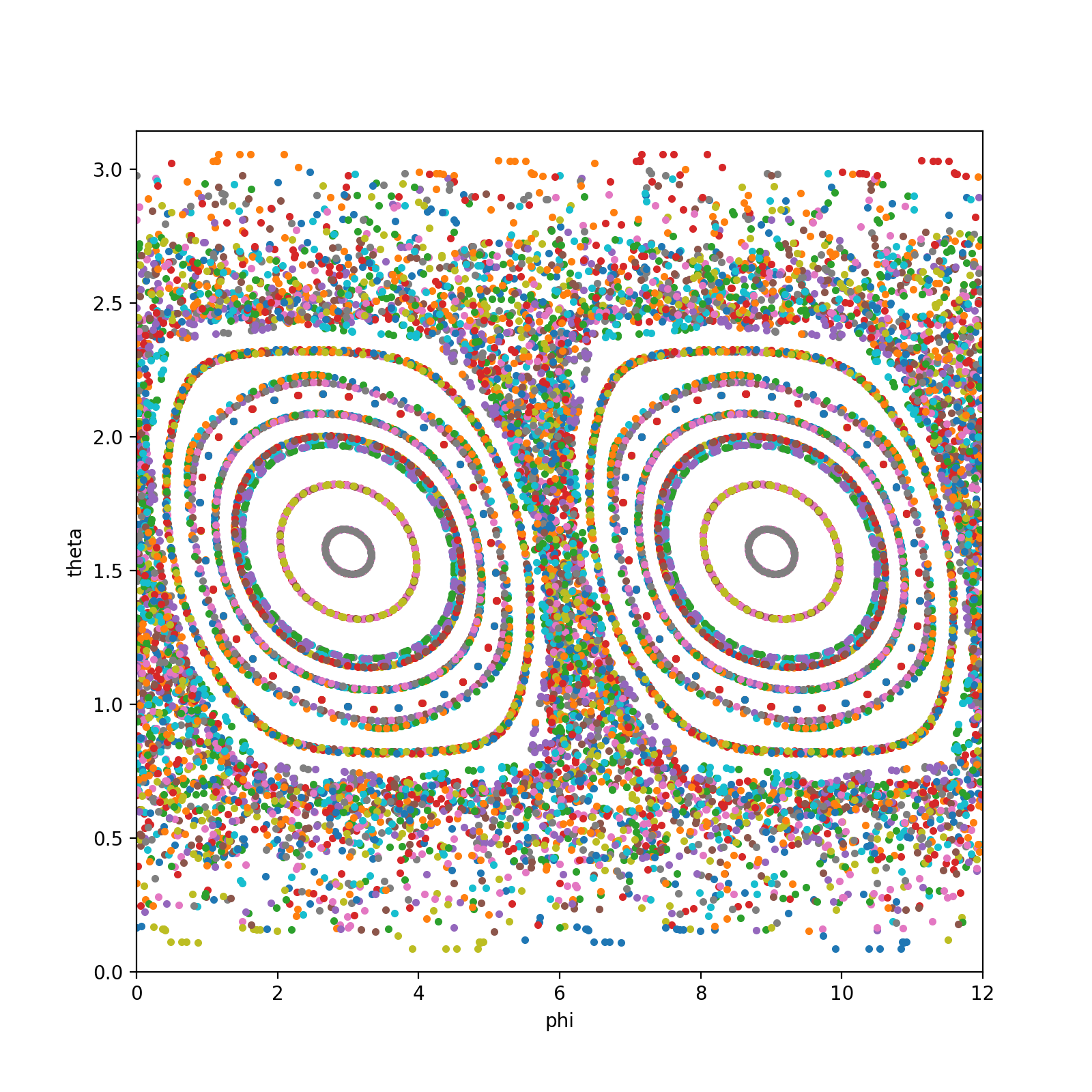}
        \caption{$a=3, \, b=1, \, \ell=1.3$}
        \label{figure_exp1pt2c}
    \end{subfigure}  
    \begin{subfigure}[b]{0.45\textwidth}
        \includegraphics[width=\textwidth]{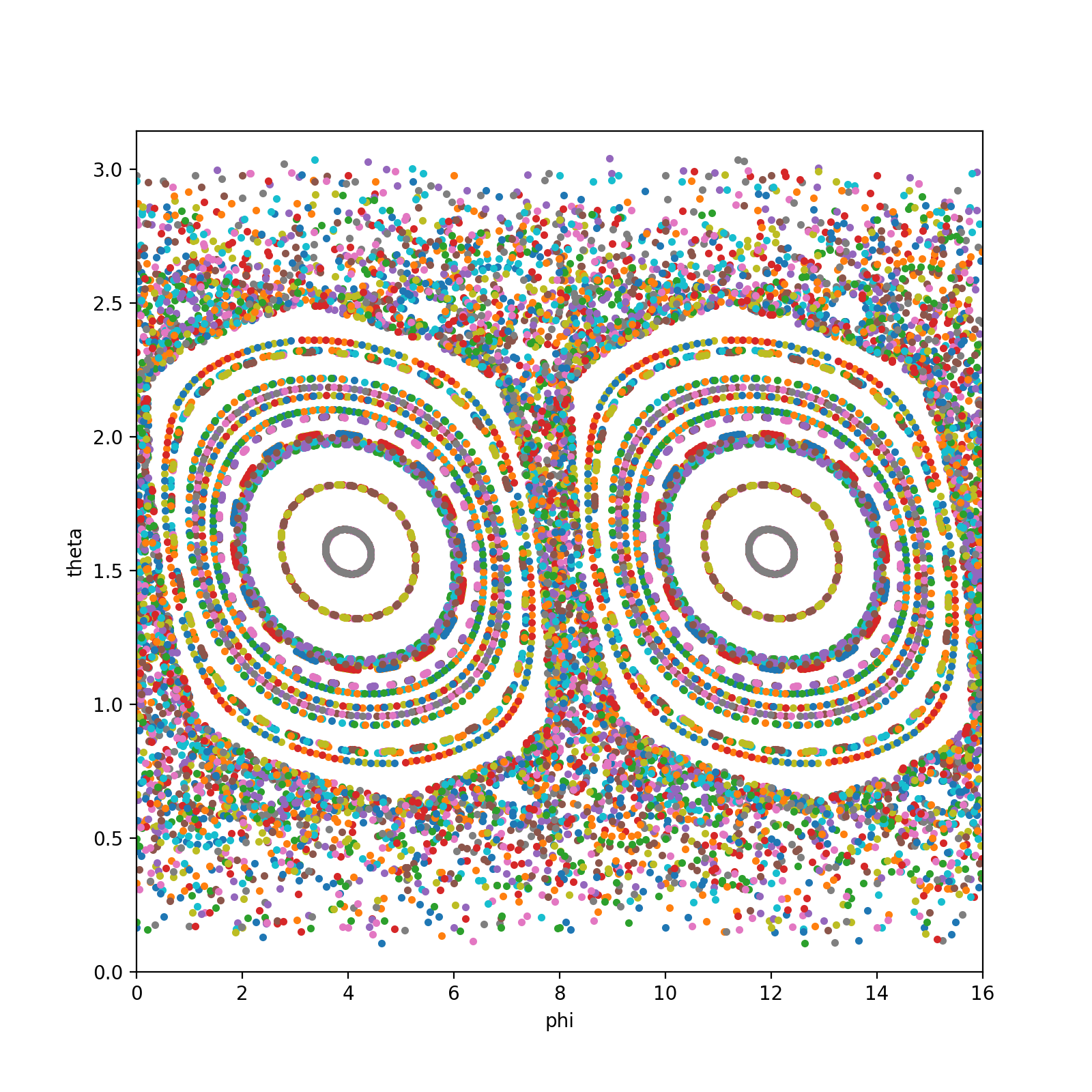}
        \caption{$a=4, \, b=1, \, \ell=1.3$}
        \label{figure_exp1pt2d}
    \end{subfigure}
    \caption{Each picture shows the first 100 iterations of various initial conditions under the coin mapping $T$ corresponding to the coin $\C \left( \mathcal{E}_{a,b},\ell \right)$, where $\mathcal{E}_{a,b}$ is ellipse with semimajor axis $a$ and semiminor axis $b$. We chose 144 different initial conditions spread evenly through the phase space. Different colours represent different orbits. The horizontal axis is the arclength parameter, and the vertical axis is the angle of incidence/reflection.}\label{figure_exp1pt2}
\end{figure}

\newpage

\appendix
\section{Moser's Theorem}\label{Appendix_Moser}

In this appendix, we give the necessary definitions, and the statement of the version of Moser's theorem on the existence of essential invariant curves for twist maps that we use in this paper. As mentioned earlier in the paper, we do not choose the strongest possible version of this result, but rather the simplest for our exposition. The version of the theorem presented here is from Sections 32 and 33 of \cite{siegel2012lectures}. 

Consider a map $F: \T \times [a,b] \to \T\times  (a',b')$ defined by
\begin{equation}
F: 
\begin{dcases}
\bar x ={}& x + \alpha (y) + f(x,y) \\
\bar y ={}& y + g(x,y)
\end{dcases}
\end{equation}
where $x \in \T$, $y \in [a,b]$, and $a' < a < b < b'$. We make the following assumptions regarding $F$:
\begin{enumerate}[(i)]
\item \label{item_app1}
$F$ is real-analytic;
\item \label{item_app2}
$F$ satisfies the twist property $\alpha'(y) \neq 0$ for all $y \in [a,b]$; and
\item \label{item_app3}
F satisfies the \emph{intersection property}: for any essential curve of the form $C=\{ (x, \zeta (x)) : x \in \T \}$ where $\zeta \in C^0(\T,[a,b])$ we have $F(C) \cap C \neq 0$. 
\end{enumerate}

Since $f$, $g$ are real-analytic, they admit holomorphic extensions to a complex neighbourhood $\U$ of $\T \times [a,b]$. We denote by $\| . \|$ the supremum norm on the space of holomorphic functions on $\U$. A condition of the theorem is that $\|f \|$ and $\| g \|$ are sufficiently small. Observe that when $\|f \|, \| g \| = 0$ then for each $y_0 \in [a,b]$, the curve $\{y=y_0\}$ is an essential $F$-invariant curve with rotation number $\omega = \alpha (y_0)$. It has been known since the works of Poincar\'e that invariant curves with a rotation number of the form $\omega = 2 \pi \, p/q $ where $p/q \in \mathbb Q$ can be destroyed by arbitrarily small perturbations satisfying conditions \eqref{item_app1}-\eqref{item_app3}. Instead, we consider frequencies which are not only irrational multiples of $2 \pi$, but are poorly approximated by rational multiples of $2 \pi$: we say that $\omega$ is \emph{Diophantine} and write $\omega \in \D$ if there are $k >0$ and $\mu \geq 2$ such that
\[
\left| \frac{\omega}{2 \pi} - \frac{p}{q} \right| \geq \frac{k}{|q|^{\mu}}
\]
for all $p, q \in \mathbb Z$ with $q \neq 0$. Then Moser's theorem tells us that any curve $\{ y = y_0\}$ that is not too close to the boundary of $\T \times [a,b]$ for which the rotation number $\omega = \alpha (y_0)$ is Diophantine survives the perturbation to small nontrivial functions $f$ and $g$, in the sense that if $\|f\|$ and $\| g \|$ are sufficiently small, then the map $F$ has an invariant essential curve $C$ that is close to $\{ y = y_0\}$ such that $F|_C$ has rotation number $\omega$. 

\begin{theorem}[Moser] \label{theorem_moser}
Let $a_0, b_0 \in \R$ with $a < a_0 < b_0 < b$. For any $\epsilon >0$ there is $\delta >0$ such that if $\| f \| + \| g \| < \delta$ then for each $y_0 \in [a_0,b_0]$ such that $\alpha(y_0) \in \D$ the map $F$ has an invariant essential curve $C$ that is $\epsilon$-close to $\{y=y_0\}$ such that $F|_C$ has rotation number $\alpha(y_0)$. Moreover these invariant curves fill up a set of positive measure in $\mathbb{T} \times [a_0,b_0]$. 
\end{theorem}

\subsection*{Acknowledgments}
The authors would like to thank Misha Bialy for proposing an interesting problem. We thank Rafael Ramírez-Ros for many helpful discussions. S.B. would like to thank Marie-Claude Arnaud, Corentin Fierobe, Anna Florio, Martin Leguil, and Jean-Pierre Marco for useful discussions. A.C. would like to thank Jacques Fejoz, Pau Martín, and Dima Turaev for useful discussions.

S.B. was supported in part by the grant RL001607 of Professor M. Guàrdia funded by the Catalan Institution for Research and Advanced Studies (ICREA), and in part by the Juan de la Cierva fellowship JDC2023-052632-I.  

A.C. was supported in part by the Grant PID-2021-122954NB-100 which was funded by MCIN/AEI/10.13039/501100011033 and “ERDF: A way of making Europe”.

\bibliographystyle{abbrv}
\bibliography{cm_ic} 
\end{document}